\documentclass[preprint,12pt]{elsarticle}

\usepackage{amsmath,amsthm,amsfonts,lipsum,fancyhdr,datetime,hyperref,mathrsfs,mathtools,enumerate}
\usepackage[margin=1in]{geometry}

\journal{J. Number Theory \textbf{227} (2021), pp. 30-93.}

\usepackage{etoolbox}
\makeatletter
\patchcmd{\ps@pprintTitle}
  {Preprint submitted to}
  {}
  {}{}
\makeatother

\usepackage{xcolor}
\usepackage{hyperref}
\definecolor{darkgreen}{rgb}{0,0.4,0}
\definecolor{BrickRed}{rgb}{0.65,0.08,0}
\hypersetup{colorlinks=true,linkcolor=blue,citecolor=red,filecolor=BrickRed,urlcolor=darkgreen}

\newtheorem{thm}{Theorem}
\newtheorem{prop}{Proposition}
\newtheorem{conjecture}{Conjecture}
\newtheorem{lemma}{Lemma}

\newtheorem{remark}{Remark}

\mathtoolsset{showonlyrefs} 

\numberwithin{equation}{section}

\numberwithin{equation}{section}

\begin{document}

\begin{frontmatter}

\author[D. T. Nguyen]{David T. Nguyen}
\address{Department of Mathematics, South Hall, University of California, Santa Barbara, CA 93106.}
\emailauthor{David.Nguyen@math.ucsb.edu}{current: d.nguyen@queensu.ca}

\title{Generalized divisor functions in arithmetic progressions: I}

\begin{abstract}
	We prove some distribution results for the $k$-fold divisor function in arithmetic progressions to moduli that exceed the square-root of length $X$ of the sum, with appropriate constrains and averaging on the moduli, saving a power of $X$ from the trivial bound. On assuming the Generalized Riemann Hypothesis, we obtain uniform power saving error terms that are independent of $k$.
	
	\bigskip
	\noindent
	We follow and specialize Y.T. Zhang's method on bounded gaps between primes to our setting. Our arguments are essentially self-contained, with the exception on the use of Deligne's work on the Riemann Hypothesis for varieties over finite fields. In particular, we avoid the reliance on Siegel's theorem, leading to some effective estimates.
	\bigskip
	\bigskip
\end{abstract}


\end{frontmatter}

\maketitle

\tableofcontents
\listoftables

\section{Introduction and statement of results}
Let $n\ge 1$ and $k\ge 1$ be integers. Let $\tau_k(n)$ denote the \textit{$k$-fold divisor function}
\begin{equation}
	\tau_k(n)=
	\sum_{n_1 n_2 \cdots n_k=n}1,
\end{equation}
where the sum runs over ordered $k$-tuples $(n_1,n_2,\dots,n_k)$ of positive integers for which $n_1 n_2 \cdots n_k=n$. Thus $\tau_k(n)$ is the coefficient of $n^{-s}$ in the Dirichlet series
\begin{equation}
	\zeta(s)^k = \sum_{n=1}^\infty \tau_k(n)n^{-s}.
\end{equation}
It is well known that the function $\tau_k$ is closely related to prime numbers. This paper is concerned with the distribution of $\tau_k(n)$ in arithmetic progressions to moduli $d$ that exceed the square-root of length of the sum, in particular, provides a sharpening of the result in \cite{WeiXueZhang2016}. We next give a brief background of the problem and present our main result.

\subsection{Survey and main result}

Towards the end of the 18th century, Gauss conjectured the celebrated Prime Number Theorem concerning the sum
\begin{equation}
	\sum_{p\le X} 1
\end{equation}
as $X$ approaches infinity, where $p$ denotes a prime. It is more convenience to count primes with weight $\log p$ instead of weight 1, c.f. Chebyshev; this leads to consideration of the sum
\begin{equation}
	\sum_{p\le X} \log p.
\end{equation}
To access the Riemann zeta function more conveniently we also count powers of primes, leading to the sum
\begin{equation}
	\sum_{\substack{p^\alpha \le X\\ \alpha\ge 1}} \log p,
\end{equation}
which is equal to the unconstrained sum over $n$
\begin{equation}
	\sum_{n\le X} \Lambda(n)
\end{equation}
where $\Lambda(n)$ is the von Mangoldt function--the coefficient of $n^{-s}$ in the series $-\zeta'(s)/\zeta(s)$. In 1837, Dirichlet considered the deep question of primes in arithmetic progression, leading him to consider sums of the form
\begin{equation} \label{eq:922}
	\sum_{\substack{n\le X\\ n\equiv a (\text{mod } d)}} \Lambda(n)
\end{equation}
for $(d,a)=1$. More generally, the function $\Lambda(n)$ is replaced by an arithmetic function $f(n)$, satisfying certain growth conditions, and we arrive at the study of the congruence sum
\begin{equation} \label{eq:1149}
	\sum_{\substack{n\le X\\ n\equiv a (\text{mod } d)}} f(n).
\end{equation}
This sum \eqref{eq:1149} is our main object of study. 

For most $f$ appearing in applications, it is expected that $f$ is distributed equally among the reduced residue classes $a (\text{mod } d)$ with $(a,d)=1$, e.g., that the sum \eqref{eq:1149} is well approximated by the average
\begin{equation} \label{eq:1234b}
	\frac{1}{\varphi(d)}
	\sum_{\substack{n\le X\\ (n,d)=1}}
	f(n)
\end{equation}
since there are $\varphi(d)$ reduced residue classes modulo $d$, where $\varphi(n)$ is the Euler's totient function. The quantity \eqref{eq:1234b} is often thought of as the `main term'. Different main terms are also considered. Thus, the study of \eqref{eq:1149} is reduced to studying the `error term'
\begin{equation} \label{eq:1152}
	\Delta(f; X, d , a) := 
	\sum_{\substack{n\le X\\ n\equiv a (\text{mod } d)}} f(n) - 
	\frac{1}{\varphi(d)}
	\sum_{\substack{n\le X\\ (n,d)=1}}
	f(n),\quad \text{for }(a,d)=1.
\end{equation}
measuring the discrepancy between the the sum \eqref{eq:1149} and the expected value \eqref{eq:1234b}. If $f$ satisfies
\begin{equation}
	f(n) \le C \tau^B(n) \log^B X
\end{equation}
for some constants $B,C>0$, which is often the case for most $f$ in applications, then a trivial bound for the discrepancy $\Delta(f; X, d , a)$ is
\begin{equation}
	\Delta(f; X, d , a) \le C' X \log^{B'} X,
\end{equation}
for some constants $B', C'>0$. The objective is then to obtain a non-trivial upper bound such as 
\begin{equation} \label{eq:nontrivialbound1}
	\Delta(f; X, d , a) \ll \frac{1}{\varphi(d)}
	\frac{X}{\log^A X},
	\quad A>0,
\end{equation}
or 
\begin{equation} \label{eq:nontrivialbound2}
	\Delta(f; X, d, a) \ll \frac{1}{\varphi(d)}
	X^{1-\delta},
	\quad 0< \delta < 1,
\end{equation}
with $d$ in a certain range depending on $X$.

For $f(n) = \Lambda(n)$, the von Mangoldt function, the clasical Siegel-Walfisz theorem implies that \eqref{eq:nontrivialbound1} holds uniformly in the range $d < \log^B X$, where $B>0$ with $A$ depending on $B$. For $f(n) = \tau_k(n)$, the estimate \eqref{eq:nontrivialbound2} is valid uniformly in the range $d\le X^{\theta_k - \epsilon}$, where the exponent of distribution $\theta_k$ are summarized in Table \ref{table:1}.
 
\begin{table}[h!]
	\caption{Only for $k=1,2,3$ is the exponent of distribution $\theta_k$ for $\tau_k(n)$ known to hold for a value larger than 1/2.}
	\label{table:1}
\begin{tabular}{ l l l }
	\hline \hline
	$\boldsymbol{k}$& $\boldsymbol{\theta_k}$& \textbf{References}\\
	\hline \hline
	$k=2$&	$\theta_2 = 2/3$& \text{Selberg, Linnik, Hooley (independently, unpublished, 1950's)};\\ 
	&&Heath-Brown (1979)
	\cite[Corollary 1, p. 409]{Heath-Brown1979}.
	\\
	$k=3$&	$\theta_3 = 1/2 + 1/230$& Friedlander and Iwaniec (1985) \cite[Theorem 5, p. 338]{FriedlanderIwaniec1985}.\\
	&	$\theta_3 = 1/2 + 1/82$& Heath-Brown (1986) \cite[Theorem 1, p. 31]{Heath-Brown1986}.\\
	&	$\theta_3 = 1/2+ 1/46$& Fouvry, Kowalski, and Michel (2015) \cite[Theorem 1.1, p. 122]{FouvryKowalskiMichel2015},\\
	&& (for prime moduli, polylog saving).\\
	\hline
	$k=4$&	$\theta_4 = 1/2$& Linnik (1961) \cite[Lemma 5, p. 197]{Linnik1961}.\\
	$k\ge4$&	$\theta_k = 8/(3k+4)$& Lavrik (1965) \cite[Teopema 1, p. 1232]{Lavrik1965}.\\
	$k=5$&	$\theta_5 = 9/20$& Friedlander and Iwaniec (1985) \cite[Theorem I, p. 273]{FriedlanderIwaniec1985a}.\\
	$k=6$&	$\theta_6 = 5/12$& Friedlander and Iwaniec (1985) \cite[Theorem II, p. 273]{FriedlanderIwaniec1985a}.\\
	$k\ge7$&	$\theta_k = 8/3k$& Friedlander and Iwaniec (1985) \cite[Theorem II, p. 273]{FriedlanderIwaniec1985a}.\\
	\hline
	$k\ge5$&	$\theta_k \ge 1/2$& Open.\\
	\hline \hline
\end{tabular}
\end{table}

In many problems in analytic number theory, it suffices to prove that \eqref{eq:nontrivialbound1} holds on average, in the sense that
\begin{equation} \label{eq:averagebound}
	\sum_{d\le X^{\theta - \epsilon}} \max_{(a,d)=1}
	|\Delta(f;X, d,a)|
	\ll \frac{X}{\log^A X}
\end{equation}
for any $\epsilon>0$, $A>0$, and some $0< \theta \le 1$. For $f(n) = \Lambda(n)$, one form of the celebrated Bombieri-Vinogradov Theorem \cite{Bombieri1965} \cite{Vinogradov1965} (1965) asserts that \eqref{eq:averagebound} holds with $\theta = 1/2$. By a general version of the Bombieri-Vinogradov theorem (see, e.g., \cite{Motohashi1976} or \cite{Wolke1973}), the bound \eqref{eq:averagebound} holds for a wide class of arithmetic functions, including $f(n) = \tau_k(n)$ for all $k$; see Table \ref{table:2} for a summary.

\begin{table}[h!] 
\caption{Known results for $\tau_k(n)$ averaged over moduli $d$, and references.}
\label{table:2}
\begin{tabular}{ l l l } 
	\hline \hline
	$\boldsymbol{k}$& $\boldsymbol{\theta_k}$& \textbf{References}\\
	\hline \hline
	$k=2$& $\theta_2=1$& Fouvry (1985) \cite[Corollaire 5, p. 74]{Fouvry1985} (exponential saving);\\
	&& Fouvry and Iwaniec (1992) \cite[Theorem 1, p. 272]{FouvryIwaniec1992}.\\
	\hline
	$k=3$&	$\theta_3 = 1/2 + 1/42$& Heath-Brown (1986) \cite[Theorem 2, p. 32]{Heath-Brown1986}.\\
	&	$\theta_3=1/2 + 1/34$& Fouvry, Kowalski, and Michel (2015) \cite[Theorem 1.2, p. 123]{FouvryKowalskiMichel2015},\\ 
	&&(for prime moduli, polylog saving).\\
	\hline
	$k\ge4$&	$\theta_k < 1/2$& Follows from the general version of Bombieri-Vinogradov theorem,\\ 
	&& see, e.g., \cite{Motohashi1976} or \cite{Wolke1973}, (polylog saving).\\
	\hline
	$k\ge4$&	$\theta_k \ge 1/2$& Open.\\
	\hline \hline
\end{tabular}
\end{table}

It is believed that \eqref{eq:averagebound} should hold with $\theta=1$ for a large class of function $f(n)$, including $\Lambda(n)$ and $\tau_k(n)$; however, going beyond $\theta>1/2$ proves to be very difficult. In the recent breakthrough work of Y. Zhang \cite{Zhang2014} on bounded gaps between primes, a crucial step is to show that, for any fixed $a\neq 0$,
\begin{equation}
	\sum_{\substack{d \in \mathcal{D}\\ d< X^{\frac{1}{2} + \frac{1}{584}}}}
	|\Delta(\Lambda; X, d,a)|
	\ll \frac{X}{\log^A X},
\end{equation}
where
\begin{equation}
	\mathcal{D}
	= \{(d,a) = 1, d\mid \prod_{p\le X^{1/1168}} p \}.
\end{equation}
In \cite{WeiXueZhang2016}, Wei, Xue, and Zhang showed that the methods of Zhang in \cite{Zhang2014} applies not only to the von Mangoldt function $\Lambda$, but also equally to the $k$-fold divisor function $\tau_k$. They proved in \cite[Theorem 1.1, p. 1664]{WeiXueZhang2016} that for any $k\ge 4$ and $a\neq 0$, we have
\begin{equation} \label{eq:1022}
	\sum_{\substack{d\in \mathcal{D}\\ d< X^{293/584}}}
	\mu(d)^2 
	\left|
	\Delta(\tau_k; X, d, a)
	\right|
	\ll X\exp (-\log^{1/2}X),
\end{equation}
where
\begin{equation} \label{eq:1019}
	\mathcal{D}
	=\{d\ge 1: (d,a)=1, (d,\prod_{p< X^{1/1168}}p) > X^{71/584} \},
\end{equation}
and the implied constant depends on $k$ and $a$. The condition on the moduli $d$ in \eqref{eq:1019} slightly allowing for $d$ to have some, but not too many, prime factors larger than $X^{1/1168}$. The error term and, more importantly, the exponent of distribution $\theta_k = 293/584 = 1/2+ 1/548$ in \eqref{eq:1022} hold uniformly in $k$. 

In the main result of this paper, we provide a sharpening of the error term in \eqref{eq:1022}, saving a power of $X$ from the trivial bound, with a constraint on the moduli $d$ not having too many very small prime factors. Actually, our arguments follow closely those of \cite{Zhang2014} in treating contribution coming from large moduli; see Section \ref{section:sketchOfProof} below for more discussion.

\begin{thm}[Main theorem]\label{thm:maintheoremtauk}
	Let 
	\begin{equation} \label{eq:varpi}
		\varpi = \frac{1}{1168}
	\end{equation}
	and
	\begin{equation} \label{eq:thetak}
		\theta_k = \min\left\{ \frac{1}{12(k+2)}, \varpi^2 \right\}.
	\end{equation}
	For $a\neq 0$, let 
	\begin{equation}
	\mathcal{D}
	=\{d\ge 1: (d,a)=1,\ |\mu(d)| =1,\ (d,\prod_{p\le X^{\varpi^2}}p) < X^\varpi,\
	\text{and}\
	(d, \prod_{p\le X^\varpi} p) > X^{71/584} \},
	\end{equation}
	where $\mu$ is the M\"obius function. Then for each $k\ge4$ we have
	\begin{equation}\label{eq:tauk}
	\sum_{\substack{d\in \mathcal{D}\\ d< X^{293/584}}}
	\left|
	\sum_{\substack{n\le X\\ n\equiv a (\text{mod } d)}} \tau_k(n) - 
		\frac{1}{\varphi(d)}
		\sum_{\substack{n\le X\\ (n,d)=1}}
		\tau_k(n)
	\right|
	\ll X^{1- \theta_k}.
	\end{equation}
	The implied constant is effective, and depends at most on $a$ and $k$.
\end{thm}

\begin{remark}
	Theorem \ref{thm:maintheoremtauk} admits several refinements. The particular choice of $\varpi=1/1168$ comes from the condition (\refeq{eq:conditionvarpi}) and, while being uniform in $k$, it is not optimal. There are certain ways to improve the numerics in Theorem \ref{thm:maintheoremtauk}, for instance using the extensive work \cite{Polymath8} of the Polymath 8 project. Though we do not focus on this aspect here, it is an open problem to replace $\varpi$ and $\theta_k$ on the right side of (\ref{eq:tauk}) by values that are as large as possible.
\end{remark}

Conditionally, if we assume the Generalized Riemann Hypothesis, or the weaker Generalized Lindel\"of Hypothesis, for Dirichlet $L$-functions we can obtain a stronger result than \eqref{eq:tauk}.

\begin{thm} \label{thm:maintheoremtaukOnLindelof}
	On the Generalized Lindel\"of Hypothesis, the estimate \eqref{eq:tauk} holds with the right side replaced by
	\begin{equation}
	X^{1-\varpi^2},
	\end{equation}
	where the $\theta_k$ power saving is replaced by a positive constant independent of $k$.
\end{thm}
\noindent
This uniform power saving is the result of sharper estimates of $L(s,\chi)^k$ on the critical line that are independent of $k$. We next present two results when we are allowed to take an extra averaging over the residue classes $a (\text{mod } d)$.

\subsection{Results on further averaging}

In a function field variant, the work of Keating, Rodgers, Roditty-Gershon, and Rudnick in \cite{KeatingRodgersRoditty-GershonRudnick2018} leads Rodgers and Soundararajan \cite[Conjecture 1]{RodgersSoundararajan2018} to the following conjecture over the integers for the variance of $\tau_k$.

\begin{conjecture}
	\label{conj:3}
	For $X,d\to \infty$ such that $\log X/\log d \to c \in (0,k)$, we have
	\begin{equation} \label{eq:2238}
		\sum_{\substack{a=1\\ (a,d)=1}}^d
		\Delta(\tau_k; X, d,a)^2
		\sim a_k(d) \gamma_k(c) X (\log d)^{k^2-1},
	\end{equation}
	where $a_k(d)$ is the arithmetic constant
	\begin{equation}
		a_k(d) = \lim_{s\to 1^+}
		(s-1)^{k^2}
		\sum_{\substack{n=1\\ (n,d)=1}}^\infty
		\frac{\tau_k(n)^2}{n^s},
	\end{equation}	
	 and $\gamma_k(c)$ is a piecewise polynomial of degree $k^2 - 1$ defined by
	 \begin{equation}
	 	\gamma_k(c) = \frac{1}{k! G(k+1)^2}
	 	\int_{[0,1]^k}
	 	\delta_c (w_1 + \cdots w_k)
	 	\Delta(w)^2 d^kw,
	 \end{equation}
	 where $\delta_c(x) = \delta(x-c)$ is a Dirac delta function centered at $c$, $\Delta(w) = \prod_{i<j} (w_i - w_j)$ is a Vandermonde determinant, and $G$ is the Barnes $G$-function, so that in particular $G(k + 1) =
	 (k - 1)! (k - 2)! \cdots 1!$.
\end{conjecture}

This conjecture is closely related to the problem of moments of Dirichlet $L$-functions and correlation of divisor sums; see, e.g, \cite{ConreyGonnek2002} and \cite{ConreyKeatingI, ConreyKeatingII, ConreyKeatingIII, ConreyKeatingIV}. In \cite{RodgersSoundararajan2018} Rodgers and Soundararajan consider smoothed sums of $\tau_k$ averaging over both $a$ and $d$, and confirm an averaged version of Conjecture \ref{conj:3} for a restricted range. Harper and Soundararajan in \cite{HarperSound} obtained a lower bound of the right order of magnitude for the average of this variance. By using the multiplicative large sieve inequality, we show that an upper bound of the same order of magnitude for this averaged variance holds.

\begin{thm} \label{thm:meanSquareResult}
	For $k\ge 4$ we have
	\begin{equation} \label{eq:meanSquareResult}
	\sum_{d\le D}
	\sum_{\substack{a=1\\ (a,d)=1}}^d
	\Delta({{\tau_k}; X, d,a})^2
	\ll
	(D + X^{1-1/6(k+2)}) X (\log X)^{k^2-1}.
	\end{equation}
\end{thm}

This result is of Barban-Davenport-Halberstam type. In forthcoming work \cite{NguyenII}, we replace the upper bound in \eqref{eq:meanSquareResult} by an asymptotic equality for the ternary divisor function $\tau_3(n)$ with the condition $(a,d)=1$ removed.

Lastly, motivated by the recent work \cite{HeathBrownLi2017} of Heath-Brown and Li in 2017, we also prove analogous bilinear estimates over hyperbolas $m\equiv an\ (\text{mod } d)$ for pairs of $\tau_k(n)$'s and $\tau_k(n)\Lambda(n)$ to moduli $d$ that can taken to be almost as large as $X^{2-\epsilon}$.

\begin{thm} \label{thm:pairOfTau_k}
	For $k\ge 4$ and any $\epsilon>0$ there holds
	\begin{equation} \label{eq:pairOfTau_k}
	\sum_{d\le D}\
	\sum_{\substack{a=1\\ (a,d)=1}}^d
	\left( \sum_{\substack{m,n\le X\\ m\equiv an ( \text{mod }d)}}
	\tau_k(m) \tau_k(n)
	- \frac{1}{\varphi(d)}
	\left(\sum_{\substack{n\le X\\ (n,d)=1}} \tau_k(n)\right)^2 \right)^2
	\ll X^{4- 1/3(k+4)}
	\end{equation}
	for any $D\le X^{2-1/3(k+2)}$.
	
	In particular, the above estimate is valid if one of the $\tau_k$ is replaced by the von Mangoldt function $\Lambda$. We have
	\begin{equation} \label{eq:pairOfTau_k2}
	\sum_{d\le D}\
	\sum_{\substack{a=1\\ (a,d)=1}}^d
	\left( \sum_{\substack{m,n\le X\\ m\equiv an (\text{mod } d)}}
	\tau_k(m) \Lambda(n)
	- \frac{X}{\varphi(d)}
	\sum_{\substack{n\le X\\ (n,d)=1}} \tau_k(n) \right)^2
	\ll X^{4- 1/3(k+4)}
	\end{equation}
	for any $D\le X^{2-1/3(k+2)}$.
\end{thm}

It might look surprising at first that the moduli in Theorems \ref{thm:pairOfTau_k} can be taken almost as large as $X^2$, but proof is in fact rather simple; the proof of Theorem \ref{thm:pairOfTau_k} follows essentially also from the large sieve inequality. 

Assuming the  Generalized Riemann Hypothesis, it might be possible to show that the estimates \eqref{eq:pairOfTau_k} and \eqref{eq:pairOfTau_k2} hold in a larger range for $d$ with right side replaced by
	\begin{equation}
	\begin{cases}
	X^{2-\delta}, & \text{for } 1\le D\le X^{1+\epsilon},\\
	X^2 D (\log X)^{k^2}, & \text{for } X^{1+\epsilon} < D \le X^2,
	\end{cases}
	\end{equation}
	for some constant $\delta>0$.
We note that the moduli $d$ in Theorems \ref{thm:meanSquareResult} and \ref{thm:pairOfTau_k} need not be smooth as in Theorems \ref{thm:maintheoremtauk} and \ref{thm:maintheoremtaukOnLindelof}.

\subsection{Acknowledgments}
I wish to express gratitude to my Ph.D. advisor Zhang YiTang for introducing me to this problem, and for his guidance and numerous encouragements. I would also like to thank the referee for their careful reading and very helpful suggestions which greatly improves the presentation of the paper.

Additionally, I am grateful to M. Ram Murty, Matthew Welsh, Carl Pomerance, Kim SungJin, Mits Kobayashi, and Garo Sarajian for helpful mathematical conversations. I'd also like to thank Birge Huisgen-Zimmermann, Jeff Stopple, and Brad Rodgers for their feedbacks and interests in this project. Further thanks to Hector Ceniceros, Dave Morrison, Mihai Putinar, Alan Krinik, Eugene Lipovetsky, Kai S. Lam, Ester Trujillo and the UCSB Graduate Scholars Program for their support in the early stage. Lastly I acknowledge the Mathematics department at UCSB, in particular Medina Price, my office mates and neighbors for comfortable working environment leading up to completion of this paper.

\section{Notation and sketch of proof}

\subsection{Notation}

$\mathbb{N} = \{1,2,3,\dots\}$.

$p$--a prime number.

$a,b,c$--integers.

$d,n,m, k, q,r,s, Q, R$--positive integers.

$\Lambda(q)$--the von Mangoldt function.

$\tau_k(q)$--the $k$-fold divisor function; $\tau_2(q) = \tau(q).$

$\varphi(n)$--the Euler's totient function.

$s=\sigma+ it$

$X$--a large real number.

$\mathcal{L} = \log X$.

$\chi(n)$--a Dirichlet character.

$e(y)$--the additive character $\exp\{2\pi i y\}.$

$e_d(y) := \exp\{2\pi i y/d\}.$

$\hat{f}$--the Fourier transform of $f$, i.e.,
$$\hat{f}(z) = \int_{-\infty}^\infty f(y) e(yz) dy.
$$

$m\equiv a (q)$ means $m\equiv a (\text{mod } q).$

$q\sim Q$ means $Q\le q < 2Q.$

$\epsilon$--any sufficiently small, positive constant, not necessarily the same in each occurrence.

$B$--some positive constant, not necessarily the same in each occurrence.

$\| \alpha \|$--means the $L^2$ norm of $\alpha=(\alpha(m))$, i.e., $$\| \alpha \| = \left(\sum_m |\alpha(m)|^2\right)^{1/2}.$$

$\chi_{N}$--the characteristic function of the subset $[N, (1+\rho) N)\subset\mathbb{R}$.

$\displaystyle\sideset{}{'}\sum_{\chi (\text{mod } d)}$--means a summation over nonprincipal characters $\chi (\text{mod } d).$

$\displaystyle\sideset{}{^*}\sum_{\chi (\text{mod } d)}$--means a summation over primitive characters $\chi (\text{mod } d).$

$\displaystyle\sum_{b (\text{mod } q)}$--means
$\displaystyle\sum_{b=1}^q$.

$\displaystyle\sideset{}{^*}\sum_{b (\text{mod } q)}$--means
$ \displaystyle\sum_{\substack{b=1\\ (b,q)=1}}^q$.

\begin{table}[h!]
	\caption{Table of parameters and their first appearance.}
	\label{table:3}
	\begin{tabular}{ l c}
		\hline \hline
		Parameters& First apprearance\\
		\hline \hline
		$\varpi = 1/1168$& \eqref{eq:varpi}\\
		$\theta_k = \min\left\{ \frac{1}{12(k+2)}, \varpi^2 \right\}$& \eqref{eq:thetak}\\
		$Q_0 = X^{1/12(k+1)}$& \eqref{eq:Q0}\\
		$D_0 = X^{\varpi^{4/3}}$& \eqref{eq:D02}\\
		$D_1 = X^\varpi$& \eqref{eq:D0}\\
		$D_2 = X^{1/2-1/12(k+1)}$& \eqref{eq:D2}\\
		$D_3 = X^{1/2 + 2\varpi}$& \eqref{eq:D3}\\
		$\mathcal{P}_0=\prod_{p\le D_0} p$& \eqref{eq:factorization2}\\
		$\mathcal{P}_1 = \prod_{p\le D_1} p$& \eqref{eq:factorization3}\\
		$\rho = X^{-\varpi}$& \eqref{eq:rho}\\
		$X_1 = X^{3/8+8\varpi}$& \eqref{eq:2.13}\\
		$X_2 = X^{1/2-4\varpi}$& \eqref{eq:2.13}\\
		\hline \hline
	\end{tabular}
\end{table}

\bigskip
We follow standard notations and write $f(X)=O(g(X))$ or $f(X) \ll g(X)$ to mean that $|f(X)| \le Cg(X)$ for some fixed constant $C$, and $f(X) = o(g(X))$ if $|f(X)|\le c(X) g(X)$ for some function $c(X)$ that goes to zero as $X$ goes to infinity. The sequences $\alpha(n)$ and $\beta(n)$ we consider are all real; in particular, the absolute value sign is not needed in several expressions.

\bigskip

\subsection{Sketch of the proof of the main theorem} \label{section:sketchOfProof}

Here, and in the rest of the paper, we fix an integer $k\ge 4$, unless specified otherwise.

To prove \eqref{eq:tauk} we follow standard practice and split the summation over moduli $d$ in into two sums: one over $d< X^{\frac{1}{2}-\delta}$ which are called small moduli and the other over $X^{\frac{1}{2}-\delta}\le d< X^{\frac{1}{2}+2\varpi}$ which are called large moduli. For small moduli, we estimate \eqref{eq:tauk} directly using the large sieve inequality together with a direct substitute for the Siegel-Walfisz condition. For the von Mangoldt function $\Lambda(n)$, the M\"obius function $\mu(n)$ is involved and, hence, the Siegel-Walfisz theorem is needed to handle very small moduli. For us, fortunately, $\tau_k$ is simpler than $\Lambda$ in that $\mu$ is absent--this feature of $\tau_k$ allows us to get a sharper bound in place of the Siegel-Walfisz theorem; see Lemma \ref{lemma:sharpsiegelwalfiszbound} below. The constant here is effective.

For large moduli, we adapt the methods of Zhang in \cite{Zhang2014} to bound the error term which goes as follows. After applying suitable combinatorial arguments, we split $\tau_k$ into appropriate convolutions as Type I, II, and III, as modeled in \cite{Zhang2014}. We treat the Type I and II in our Case (b), Type III in our Case (c), and Case (a) corresponds to a trivial case which we treat directly. The main ingredients in Case (b) are the dispersion method and Weil bound on Kloosterman sums. The Case (c) depends crucially on the factorization $d=qr$ of the moduli to Weil shift a certain incomplete Kloosterman sum to the modulus $r$. The shift modulo this $r$ then induces a Ramanujan sum, which is known to have better than square-root cancellation. This allows for a saving of a power of $r$, and since $d$ is a multiple of $r$, and $d$ is less than $X$, this saves a small power of $X$ from the trivial bound.

\section{Preliminary lemmas}

We collect here lemmas that shall be used to prove our theorems. Some lemmas are standard and we quote directly from the literature.

\begin{lemma} \label{lemma:tauBound}
	For any $\epsilon >0$ we have
	\begin{equation} \label{eq:CrudeBoundOnTauk}
		\tau_j(n) \ll n^\epsilon.
	\end{equation}
\end{lemma}
\begin{proof}
	See \cite[Equation (1.81)]{IwaniecKowalski2004}.
\end{proof}

\begin{lemma}\label{lemma:primitivesum}
	Let $\gamma$ be an arithmetic function. If $\chi (\text{mod } d)$ is nonprincipal, then there exists a unique $q|d$, $q>1$, and a unique primitive character $\chi^* (\text{mod } q)$, such that, with $r=d/q$, 
	\begin{equation} \label{eq:primitivesum}
	\sum_{n} \gamma(n)\chi(n)
	= \sum_{(n,r)=1} \gamma(n)\chi^*(n).
	\end{equation}
\end{lemma}
\begin{proof}
	See, e.g., \cite[Section 5]{Davenport} for definition of characters and proofs.
\end{proof}

In reducing nonprincipal characters, which may have not too small moduli, to primitive characters for the application of large sieve inequality, very small moduli of the primitive characters may occur. We treat contributions from those small moduli via the following lemma.

\begin{lemma}\label{lemma:tinymoduli}
	Let $\chi$ be a primitive character (mod $d$). For $d< X^{1/3(k+1)}$ we have
	\begin{equation} \label{eq:dsmallpowerofx}
	\sum_{n\le X} \tau_k(n)\chi(n) 
	\ll X^{1-\frac{1}{3(k+2)}}.
	\end{equation}
\end{lemma}
\begin{proof} Decompose the interval $[1,X]$ in to dyadic intervals of the form $[N,2N)$. Denote by
	\begin{equation}
	\psi(\chi)=\sum_{n\sim N} \tau_k(n)\chi(n).
	\end{equation}
	Let $0<\eta<1$ be a parameter to be specified latter (see \eqref{eq:eta} below). Let $f(x)$ be a function of $C^\infty(-\infty,\infty)$ class such that $0\le f(y)\le 1$,
	\begin{equation}
	f(y) = 1\quad \text{ if }\quad N\le y \le 2N,
	\end{equation}
	\begin{equation}
	f(y) = 0\quad \text{ if }\quad 
	y\notin [N-N^\eta, 2N+N^\eta],
	\end{equation}
	and obeying the derivative bound
	\begin{equation} \label{eq:fjDerivativeBound}
	f^{(j)}(y) \ll N^{-j\eta},\ j\ge 1,
	\end{equation}
	where the implied constant depends on $\eta$ and $j$ at most.
	Let
	\begin{equation} \label{eq:S*}
	\psi^*(\chi)=\sum_{n=1}^\infty \tau_k(n) \chi(n) f(n).
	\end{equation}
	By \eqref{eq:CrudeBoundOnTauk}, we have
	\begin{equation}
	\psi^*(\chi) - \psi(\chi)
	= \sum_{N-N^\eta\le n\le N}\tau_k(n)\chi(n)
	+ \sum_{2N\le n\le 2N+N^\eta}\tau_k(n)\chi(n)
	\ll N^{\eta+\epsilon}
	\end{equation}
	for any $\epsilon>0$. Let
	\begin{equation} \label{eq:defofG}
	F(s)=\int_0^\infty f(x)x^{s-1} dx
	\end{equation}
	be the Mellin transform of $f(x)$. The function $F(s)$ is absolutely convergent for $\sigma>0$ with inverse Mellin transform
	\begin{equation} \label{eq:inverseMellin}
	f(x) = \frac{1}{2\pi i}
	\int_{(2)}
	F(s) x^{-s} ds,
	\end{equation}
	where $\int_{(c)}$ denotes the integration $\int_{c-i\infty}^{c+i\infty}$ over the vertical line $c+it$ where $t$ runs from $-\infty$ to $\infty$. Substituting \eqref{eq:inverseMellin} into \eqref{eq:S*} and changing the order of summation and integration, we get
	\begin{align}
	\psi^*(\chi) 
	&= \sum_{n=1}^\infty \tau_k(n) \chi(n) \left(\frac{1}{2\pi i} \int_{(2)} F(s) n^{-s} ds \right)\\
	&= \frac{1}{2\pi i} \int_{(2)} F(s) \left(\sum_{n=1}^\infty \tau_k(n) \chi(n) n^{-s} \right) ds\\
	\label{eq:line2}&= \frac{1}{2\pi i} \int_{(2)} F(s) L(s,\chi)^k ds,
	\end{align}
	where $L(s,\chi)=\sum_{n=1}^\infty \chi(n) n^{-s}$ is the Dirichlet series for $\chi$. Since the function $L(s,\chi)$, and thus, $F(s)L(s,\chi)^k$ has no poles in $\sigma\ge 0$, we may move the line of integration in \eqref{eq:line2} from $\sigma=2$ to $\sigma=1/2$ and obtain
	\begin{equation} \label{eq:fmellin}
	\psi^*(\chi) =
	\frac{1}{2\pi i} \int_{(\frac{1}{2})} F(s) L(s,\chi)^k ds.
	\end{equation}
	
	We next estimate this integral by bounding the integrand and splitting the line of integration into two parts, over $|t|<T$ and $|t|\ge T$, then choosing $T$ suitably (see \eqref{eq:T} below). For $\sigma=1/2$, we have the convexity bound; see, e.g., \cite[Theorem 5.23]{IwaniecKowalski2004},
	\begin{equation}\label{eq:Lkbound}
	|L(s,\chi)|^k \ll d^{k/4} |s|^k.
	\end{equation}
	We next obtain upper bound for $F(s)$. On the line $\sigma=1/2$, we have, by definitions of $F(s)$ and $f(x)$,
	\begin{equation}\label{eq:Gcrudebound}
	F(s)
	= \int_{N-N^\eta}^{2N+N^\eta} f(x) x^{s-1} dx
	\le \int_{N-N^\eta}^{2N+N^\eta} x^{-1/2} dx
	\ll N^{1/2}.
	\end{equation}
	This bound is sufficient for bounding small $|t|$ in \eqref{eq:fmellin}, but too large for  $|t|$ large. To bound contribution from large $|t|$ we fix an
	\begin{equation} \label{eq:ell}
	\ell > k+1
	\end{equation}
	and apply integration by parts $\ell$ times to $F(s)$:
	\begin{align}
	F(s)&=(-1)^\ell \frac{1}{s(s+1)\cdots(s+\ell+1)}\int_0^\infty f^{(\ell)}(x) x^{s+\ell-1} dx.
	\end{align}
	Hence, by the derivative bound \eqref{eq:fjDerivativeBound}, $F(s)$ is bounded by
	\begin{align}\label{eq:Gbetterbound}
	|F(s)|
	\ll \frac{1}{|s|^\ell} N^{-\ell\eta+1/2+\ell}
	\ll \frac{1}{|s|^\ell} N^{(1-\eta)\ell+1/2}.
	\end{align}
	This bound allows us to save an arbitrary negative power of $|s|$; we will use this bound for large $|t|$. 
	
	We now split the integral in \eqref{eq:fmellin} into two and estimate each part individually. Let $s=1/2+it$. For $T>2$, we can write $\psi^*(\chi)$ in \eqref{eq:fmellin} as
	\begin{equation}
	\psi^*(\chi) = \frac{1}{2\pi i}\int_{|t|<T} F(s)L(s,\chi)^k  ds
	+ \frac{1}{2\pi i}\int_{|t|\ge T} F(s)L(s,\chi)^k ds.
	\end{equation}
	By \eqref{eq:Lkbound} and \eqref{eq:Gcrudebound}, the first term on the right side is
	\begin{equation}\label{eq:smalltcontribution}
	\frac{1}{2\pi i}\int_{|t|<T} 
	F(s) L(s,\chi)^k ds
	\ll N^{1/2}d^{k/4}\int_{-T}^T |1/2+it|^k dt
	\ll N^{1/2}d^{k/4}T^{k+1};
	\end{equation}
	while by \eqref{eq:Lkbound}, \eqref{eq:Gbetterbound}, and \eqref{eq:ell}, the second term on the right side is
	\begin{equation}\label{largetcontribution}
	\frac{1}{2\pi i}\int_{|t|\ge T} 
	F(s) L(s,\chi)^k  ds
	\ll N^{(1-\eta)\ell+1/2}d^{k/4}\int_T^\infty t^{k-\ell} dt
	\ll N^{(1-\eta)\ell +1/2}d^{k/4}T^{k-\ell+1}.
	\end{equation}
	Hence we choose 
	\begin{equation} \label{eq:T}
	T=N^{1-\eta},
	\end{equation}
	so that the contributions in \eqref{eq:smalltcontribution} and (\ref{largetcontribution} are of the same order. With this choice of $T$, we get
	\begin{equation}
	\psi^*(\chi) \ll N^{(1-\eta)(k+1)+1/2}d^{k/4}.
	\end{equation}
	
	It remains to specify $\eta$. Let
	\begin{equation} \label{eq:eta}
	\eta = 1-\frac{1}{3(k+1)},
	\end{equation}
	so that $(1-\eta)(k+1)=1/3$. Thus
	\begin{equation}
	\psi^*(\chi) 
	\ll N^{5/6}d^{k/4}
	\end{equation}
	and
	\begin{equation}
	\psi(\chi) 
	\ll N^{5/6}d^{k/4} + N^{\eta+\epsilon}.
	\end{equation}
	Summing over $N=X2^{-k}$, we find that the left side of \eqref{eq:dsmallpowerofx} is bounded by
	\begin{equation}
	\ll X^{5/6+\epsilon} d^{k/4}
	+ X^{1-\frac{1}{3(k+1)}+\epsilon}.
	\end{equation}
	Thus, if $d<X^{1/3(k+1)}$, then the above estimate is
	\begin{equation}
	\ll X^{1-\frac{1}{3(k+1)}+\epsilon}
	\ll X^{1-\frac{1}{3(k+2)}}
	\end{equation}
	for small enough $\epsilon$. This gives the estimate \eqref{eq:dsmallpowerofx}.
\end{proof}

\begin{lemma}\label{lemma:nonprincipalsum}
	Let $\gamma$ be an arithmetic function. For $(a,d)=1$ we have
	\begin{equation} \label{eq:nonprincipalsum2}
	\Delta(\gamma;X,d,a)
	=\frac{1}{\varphi(d)} \sideset{}{'}\sum_{\chi (\text{mod } d)} 
	\overline{\chi}(a)
	\left(\sum_{n\le X} \gamma(n)\chi(n)\right).
	\end{equation}
\end{lemma}
\begin{proof} 
	By the orthogonality condition
	\begin{equation}
		\frac{1}{\varphi(d)} \sum_{\chi (\text{mod } d)} \overline{\chi}(a) \chi(n)
		=
		\begin{cases}
		1,& \text{ if } n\equiv a (d),\\
		0,& \text{ otherwise},
		\end{cases}
	\end{equation}
	we may write
	\begin{equation} \label{eq:sumGamma}
		\sum_{\substack{n\le X\\ n\equiv a (d)}}
		\gamma(n)
		=\sum_{n\le X} \gamma(n)
		\left(\frac{1}{\varphi(d)} \sum_{\chi (\text{mod } d)} \overline{\chi}(a)\chi(n) \right)
		= \frac{1}{\varphi(d)} \sum_{\chi (\text{mod } d)} \overline{\chi}(a) 
		\left(\sum_{n\le X} \gamma(n)\chi(n)\right).
	\end{equation}
	If $\chi$ (mod $d$) is principal, then
	\begin{equation}
		\begin{cases}
		\overline{\chi}(a) = 1,\\
		\displaystyle\sum_{n\le X} \gamma(n) \chi(n) 
		= \displaystyle\sum_{\substack{n\le X\\ (n,d)=1}}
		\gamma(n).
		\end{cases}
	\end{equation}
	Hence the contribution from the principal character gives the main term in \eqref{eq:sumGamma} and the discrepancy $\Delta(\gamma;X,d,a)$ is given by a sum over nonprincipal characters. This gives \eqref{eq:nonprincipalsum2}.
\end{proof}

The next lemma is the well-known multiplicative large sieve inequality.

\begin{lemma} \label{lemma:largesievemultiplicative}
	Let $\chi$ be a primitive character mod $q$. For $a(n)$ a sequence of complex numbers, we have
	\begin{equation} \label{eq:largesieveineq}
		\sum_{q\le Q}
		\sideset{}{^*}\sum_{\ \chi(\text{mod } q)}
		\left|\sum_{n\le N} a(n)\chi(n) \right|^2
		\ll
		(Q^2+N)\sum_{n\le N} |a(n)|^2.
	\end{equation}
\end{lemma}
\begin{proof}
	See \cite[Theorem 7.13]{IwaniecKowalski2004}.
\end{proof}

The next lemma is a truncated Poisson formula.

\begin{lemma} \label{lemma:7}
	Suppose that $\eta^*>1$ and $X^{1/4}<M<X^{2/3}$. Let $f$ be a function of $C^{\infty}(-\infty,\infty)$ class such that $0\le f(y)\le 1$,
	\begin{equation}
	f(y)=1\quad \text{if}\quad M\le y\le \eta^*M,
	\end{equation}
	\begin{equation} 
	f(y)=0\quad \text{if}\quad y\notin[(1-M^{-\epsilon})M,(1+M^{-\epsilon}\eta^*)M],
	\end{equation}
	and
	\begin{equation}
	f^{(j)}(y) \ll M^{-j(1-\epsilon)},\quad j\ge 1,
	\end{equation}
	with the implied constant depending on $\epsilon$ and $j$ at most. Then we have
	\begin{equation}
	\sum_{m\equiv a (d)} f(m)
	= \frac{1}{d} \sum_{|h|<H}
	\hat{f}(h/d) e_d(-ah) + O(d^{-1})
	\end{equation}
	for any $H\ge dM^{-1+2\epsilon}$, where $\hat{f}$ is the Fourier transform of $f$.
\end{lemma}
\begin{proof}
	See \cite[Lemma 2]{BFI1986}.
\end{proof}

\begin{lemma} \label{lemma:8}
	Suppose that $1\le N<N'<2x$, $N'-N > X^\epsilon d$, and $(c,d)=1$. Then for $j,\nu\ge 1$ we have
	\begin{equation} \label{eq:lemma81}
	\sum_{\substack{N\le n\le N'\\}}
	\tau_j(n)^\nu
	\ll (N'-N) \mathcal{L}^{j^\nu-1},
	\end{equation}
	and
	\begin{equation} \label{eq:lemma82}
	\sum_{\substack{N\le n\le N'\\ n\equiv c(d)}}
	\tau_j(n)^\nu
	\ll \frac{N'-N}{\varphi(d)} \mathcal{L}^{j^\nu-1}.
	\end{equation}
	The implied constants depending on $\epsilon, j$, and $\nu$ at most.
\end{lemma}
\begin{proof}
	See \cite[Theorem 2]{Shiu1980}.
\end{proof}

In the next lemma we verify a substitute for the ``Siegel-Walfisz" condition.

\begin{lemma} \label{lemma:sharpsiegelwalfiszbound}
	Let $\beta=\beta_{i_1}*\cdots * \beta_{i_\ell}$, $1\le i_1\le i_2\le \cdots \le i_\ell \le k$,  and $\beta_j=\chi_{N_j}$, with $N:=N_{i_1}N_{i_2}\cdots N_{i_\ell}\gg X^\kappa$ for some constant $\kappa>0$. For $\chi$ a primitive character modulo $r\ll X^{\kappa}$, we have
	\begin{equation} \label{eq:sharpsiegelwalfiszbound}
		\sum_n \beta(n) \chi(n)
		\ll X^{-\kappa/12} N.
	\end{equation}
\end{lemma}

\begin{proof}
	We first verify \eqref{eq:sharpsiegelwalfiszbound} for a single $\beta=\beta_i$. For the general case, it suffices to check that if $\beta_i$ and $\beta_j$ satisfy \eqref{eq:sharpsiegelwalfiszbound}, then so does their convolution $\beta_i*\beta_j$.
	
	Let $\beta=\chi_{N_i}$ $N=N_i\gg X^{\kappa}$. We proceed analogously as to the proof of Lemma \ref{lemma:tinymoduli}. Let $f(x)$ be a function of $C^\infty(-\infty,\infty)$ class such that $0\le f(y)\le 1$,
	\begin{equation}
	f(y) = 1\quad \text{ if }\quad N\le y \le (1+\rho)N,
	\end{equation}
	\begin{equation}
	f(y) = 0\quad \text{ if }\quad 
	y\notin [N-N^{11/12}, (1+\rho)N+N^{11/12}],
	\end{equation}
	and obeying the derivative bound
	\begin{equation} \label{eq:fjDerivativeBound2}
	f^{(j)}(y) \ll N^{-11j/12},\ j\ge 1,
	\end{equation}
	where the implied constant depends on $j$. Let
	\begin{equation}
		F(s) = \int_0^\infty f(x) x^{s-1} dx
	\end{equation}
	denote the Mellin transform of $f(x)$. Let
	\begin{equation}
		\psi(\chi) = \sum_n \beta(n) \chi(n)
	\end{equation}
	and
	\begin{equation}
		\psi^*(\chi) = \sum_n \beta(n) \chi(n) f(n).
	\end{equation}
	Analogously, we have
	\begin{equation}
		\psi^*(\chi) - \psi(\chi)
		= \sum_{N-N^{11/12}\le n\le N} \chi(n)
		+ \sum_{(1+\rho)N\le n\le (1+\rho)N+N^{11/12}} \chi(n)
		\ll N^{11/12}
	\end{equation}
	and
	\begin{equation}
		|L(1/2+it,\chi)| \ll r^{1/4} |s|.
	\end{equation}
	Thus,
	\begin{align}
		\psi^*(\chi)
		&= \frac{1}{2\pi i} \int_{(\frac{1}{2})}
		F(s) L(s,\chi) ds\\
		&=\frac{1}{2\pi i} \int_{|t|<N^{1/12}} F(1/2+it)L(1/2+it,\chi)  ds\\
		&\quad + \frac{1}{2\pi i}\int_{|t|\ge N^{1/12}} F(1/2+it)L(1/2+it,\chi) ds\\
		&\ll N^{1/2} r^{1/4} N^{1/6} + N^{1/2} r^{1/4} N^{1/6}
		\ll N^{2/3} r^{1/4}.
	\end{align}
	Assume $r\ll X^\kappa$. We deduce
	\begin{align}
		\psi(\chi)
		&\ll N^{11/12} + N^{2/3} r^{1/4}
		\ll N \frac{1}{N^{1/12}}
		+ N \frac{r^{1/4}}{N^{1/3}}
		\ll N X^{-\kappa/12} + N \frac{X^{\kappa/4}}{X^{\kappa/3}}
		\ll N X^{-\kappa/12}.
	\end{align}
	
	Now assume $\beta_i$ and $\beta_j$ satisfy \eqref{eq:sharpsiegelwalfiszbound} with $N:=N_iN_j\gg X^{\kappa}$. Write $N_i=X^{\kappa_i}$ and $N_j=X^{\kappa_j}$ so that $\kappa_i + \kappa_j \ge \kappa$. Since $\beta_i$ and $\beta_j$ satisfy \eqref{eq:sharpsiegelwalfiszbound}, we have
	\begin{equation}
		\sum_n \beta_i(n) \chi(n)
		\ll N_i X^{-\kappa_i/12}
	\end{equation}
	and 
	\begin{equation}
		\sum_n \beta_j(n) \chi(n)
		\ll N_j X^{-\kappa_j/12}.
	\end{equation}
	Thus, writing $n$ as $mn$ and separate variables, we get
	\begin{equation}
		\sum_n \beta_i*\beta_j(n) \chi(n)
		= \sum_m \beta_i(m) \chi(m)
		\sum_n \beta_i(n) \chi(n)
		\ll N_i X^{-\kappa_i/12}
		N_j X^{-\kappa_j/12}
		\ll N X^{-\kappa/12}.
	\end{equation}
	This completes the proof of Lemma \ref{lemma:sharpsiegelwalfiszbound}.
\end{proof}

\begin{lemma} \label{lemma:10}
	Let $\beta$ be given as in \eqref{eq:beta}, with $N$ given in \eqref{eq:betaN} satisfying \eqref{eq:betaNbound}. Assume $R\le X^{-\varpi/6}N$. Then for any $q\ge 1$ and $(r,\ell)$, we have
	\begin{equation} \label{eq:946}
	\sum_{r\sim R}\
	\displaystyle\sideset{}{^*}\sum_{\ell (\text{mod } r)}
	\left(
	\sum_{\substack{n\equiv \ell (r)\\ (n,q)=1}} \beta(n) 
	- \frac{1}{\varphi(r)} \sum_{(n,qr)=1} \beta(n)
	\right)^2
	\ll \tau(q)^B N^2 X^{-{\frac{\varpi}{12}}}.
	\end{equation}
\end{lemma}
\begin{proof}
	By M\"obius inversion, the condition $(n,q)=1$ may be removed at the cost of removing the $\tau(q)^B$ factor on the right side of \eqref{eq:946}; see, e.g., \cite[p. 21-22]{FriedlanderIwaniec2013}. Thus it suffices to show
	\begin{equation} \label{eq:919}
	\sum_{r\sim R}\
	\displaystyle\sideset{}{^*}\sum_{\ell (\text{mod } r)}
	\Delta(\beta; X, r, \ell)^2
	\ll N^2 X^{-{\varpi/12}}.
	\end{equation}
	By \eqref{eq:nonprincipalsum2}, we have
	\begin{align}
	\Delta(\beta; X, r, \ell)^2
	&= \frac{1}{\varphi(r)^2}
	\left|
	\displaystyle\sideset{}{'}\sum_{\chi (\text{mod } r)}
	\overline{\chi}(a)
	\left(
	\sum_{n\le X} \beta(n) \chi(n)
	\right)
	\right|^2\\
	&=\frac{1}{\varphi(r)^2}
	\displaystyle\sideset{}{'}\sum_{\chi_1 (\text{mod } r)}
	\overline{\chi_1}(a)
	\left(
	\sum_{n\le X} \beta(n) \chi_1(n)
	\right)
	\displaystyle\sideset{}{'}\sum_{\chi_2 (\text{mod } r)}
	\chi_2(a)
	\overline{\left(
		\sum_{n\le X} \beta(n) \chi_2(n)
		\right)}.
	\end{align}
	Summing over primitive $\ell (\text{mod }r)$ and changing the order of summation, we get
	\begin{equation}
	\displaystyle\sideset{}{^*}\sum_{\ell (\text{mod } r)}
	\Delta(\beta; X, r, \ell)^2
	= \frac{1}{\varphi(r)^2}
	\displaystyle\sideset{}{'}\sum_{\chi_1 (\text{mod }r)}\
	\displaystyle\sideset{}{'}\sum_{\chi_2 (\text{mod }r)}
	{\left(
		\sum_{n\le X} \tau_k(n) \chi_1(n)
		\right)}
	\overline{\left(
		\sum_{n\le X} \tau_k(n) \chi_2(n)
		\right)}
	\displaystyle\sideset{}{^*}\sum_{a (\text{mod } r)}
	\overline{\chi_1}(a) \chi_2(a).
	\end{equation}
	By the orthogonality relation
	\begin{equation} \label{eq:928}
	\frac{1}{\varphi(r)} 
	\displaystyle\sideset{}{^*}\sum_{a (\text{mod } r)}
	\overline{\chi_1}(a) \chi_2(a) 
	= \begin{cases}
	1,& \text{if } \chi_1= \chi_2,\\
	0,& \text{if }\chi_1\neq \chi_2,
	\end{cases}
	\end{equation}
	this becomes
	\begin{equation}
	\displaystyle\sideset{}{^*}\sum_{\ell (\text{mod } r)}
	\Delta(\beta; X, r, \ell)^2
	= \frac{1}{\varphi(r)}
	\left|
	\displaystyle\sideset{}{'}\sum_{\chi (\text{mod } r)}
	{\left(
		\sum_{n\le X} \beta(n) \chi(n)
		\right)}
	\right|^2.
	\end{equation}
	We now reduce to primitive characters as in the proof of Proposition \ref{prop:mediummoduli}. By Lemma \ref{lemma:primitivesum}, we have
	\begin{equation}
	\sum_{r\sim R}\
	\displaystyle\sideset{}{^*}\sum_{\ell (\text{mod } r)}
	\Delta(\beta; X, r, \ell)^2
	\ll 
	\log \mathcal{L}
	\sum_{s\le R}
	\frac{1}{s}
	\left(
	\sum_{1<q\le R/s}
	\frac{1}{q}
	\left|\
	\displaystyle\sideset{}{^*}\sum_{\chi (\text{mod } q)}
	\left(
	\sum_{n\le X} \beta(n) \chi(n)
	\right)
	\right|^2
	\right).
	\end{equation}
	We apply the Siegel-Walfisz condition to bound contribution coming from tiny moduli. By Lemma \ref{lemma:sharpsiegelwalfiszbound}, we get, for $1< q\le X^{\varpi/6}$,
	\begin{equation}
	\sum_{1<q\le X^{\varpi/12}}
	\frac{1}{q}
	\left|\
	\displaystyle\sideset{}{^*}\sum_{\chi (\text{mod } q)}
	\left(
	\sum_{n\le X} \beta(n) \chi(n)
	\right)
	\right|^2
	= \sum_{1<q\le X^{\varpi/6}}
	\frac{1}{q}
	\left(
	\varphi(q) X^{-3/96 - 2\varpi/3}N
	\right)^2
	\ll N^2 X^{-1/16}.
	\end{equation}
	Assume $X^{\varpi/6} \ll Q\ll R$. By the large sieve inequality \eqref{eq:largesieveineq} and the bound \eqref{eq:lemma81}, we have
	\begin{equation}
	\frac{1}{Q}
	\sum_{q\sim Q}
	\left| \
	\displaystyle\sideset{}{^*}\sum_{\chi (\text{mod } q)}
	\left(
	\sum_{n\le X} \beta(n) \chi(n)
	\right)
	\right|^2
	\ll \frac{1}{Q}(Q^2 + N)
	\left(\sum_{n \le X} \beta(n) \right)^2
	\ll \left(Q + \frac{N}{Q} \right) N \mathcal{L}^{k-1}.
	\end{equation}
	For $R\le X^{-\varpi/6}N$, this leads to \eqref{eq:919}.
\end{proof}

The next lemma restricts $d$ to moduli that have `well-factorable' property.

\begin{lemma}[Factorization lemma] \label{lemma:4}
	Write
	\begin{equation} \label{eq:D3}
		D_3 = X^{1/2 + 2\varpi}.
	\end{equation}
	Suppose $d$ is square-free such that 
	\begin{equation} \label{eq:factorization1}
	D_2 <d < D_3,
	\end{equation} 
	\begin{equation} \label{eq:factorization2}
		(d,\mathcal{P}_0) < X^\varpi,\quad
		\varpi = 1/1168,
	\end{equation}
	and
	\begin{equation}  \label{eq:factorization3}
		(d,\mathcal{P}_1) > X^{1/8-4\varpi}
	\end{equation} 
	Then, for any $R^*$ satisfying 
	\begin{equation} \label{eq:R*1}
		X^{2\varpi} \le R^*\le X^{45\varpi}
	\end{equation}
	or
	\begin{equation} \label{eq:R*2}
		X^{3/8+7\varpi} \le R^* \le X^{1/2-2\varpi},
	\end{equation}
	there is a factorization $d=qr$ such that $X^{-\varpi}R^* <r < R^*$ and $(q,\mathcal{P}_0) =1$.
\end{lemma}

\begin{proof}
	 Since $d$ is square-free, we may write $d$ as $d=d_0d_1 d_2$ with
	\begin{equation}
		d_0 = (d,\mathcal{P}_0)
	\end{equation}
	\begin{equation} \label{eq:D0}
		d_1
		=\frac{(d,\mathcal{P}_1)}{(d,\mathcal{P}_0)}
		= \prod_{j=1}^n p_j, \quad D_0< p_1 < p_2 < \dots < p_n < D_1,\quad n\ge 2,
	\end{equation}
	and
	\begin{equation}
		d_2 = \prod_{\substack{p|d\\ p> D_1}} p.
	\end{equation}
	We have $d_0< X^\varpi$. By the first inequality in \eqref{eq:R*1}, $R^* \ge X^{2\varpi}$, and there is an $n'<n$ such that
	\begin{equation}
		d_0 \prod_{j=1}^{n'} p_j < R^*
		\quad\text{and}\quad
		d_0 \prod_{j=1}^{n'+1} p_j \ge R^*.
	\end{equation}
	Similarly, by \eqref{eq:factorization2}, we have
	\begin{equation}
		d_2 = \frac{d}{(d,\mathcal{P}_1)} < X^{3/8+6\varpi}.
	\end{equation}
	By the first inequality in \eqref{eq:R*2}, $R^* \ge X^{3/8+7\varpi}$, and there is an $n''<n$ such that
	\begin{equation}
		d_2 \prod_{j=1}^{n''} p_j < R^*
		\quad\text{and}\quad
		d_2 \prod_{j=1}^{n''+1} p_j \ge R^*.
	\end{equation}
	The assertion follows by choosing
	\begin{equation}
		r = d_0r_i,\ i=1,2,
	\end{equation}
	where
	\begin{equation}
		r_1 = d_0 \prod_{j=1}^{n'} p_j
		\quad\text{and}\quad
		r_2 = d_0d_2 \prod_{j=1}^{n''} p_j
	\end{equation}
	and noting that $r_i\ge R^*/p_{n'+1}$.
\end{proof}

\begin{lemma}(\cite[Lemma 9]{Zhang2014} \label{lemma:9})
	Suppose that $H,N\ge 2$ and $(c,d)=1$. Then we have\begin{equation} \label{eq:3.8}
	\sum_{\substack{n\le N\\ (n,d)=1}}
	\min\{H,\|c\overline{n}/d \|^{-1} \}
	\ll (dN)^\epsilon (H+N),
	\end{equation}
	where $\overline{n}/d$ means $a/d (\text{mod }1)$ with $an \equiv 1 (\text{mod } d)$.
\end{lemma}

We quote a crucial bound on an incomplete Kloosterman sum obtained in \cite[Lemma 11]{Zhang2014}.

\begin{lemma}[{\cite[Lemma 11]{Zhang2014}}] \label{lemma:11}
	Suppose that $N\ge 1$, $d_1d_2>10$, and $|\mu(d_1)| = |\mu(d_2)|=1$. Then for any $c_1$, $c_2$, and $\ell$, we have
	\begin{equation} \label{eq:3.9}
		\sum_{\substack{n\le N\\ (n,d_1)=1\\ (n+\ell,d_2)=1}}
		e\left(\frac{c_1\overline{n}}{d_1} + \frac{c_2\overline{(n+\ell)}}{d_2} \right)
		\ll (d_1d_2)^{1/2} \tau(d_1d_2)
		+ \frac{(c_1,d_1) (c_2,d_2) (d_1,d_2)^2 N}{d_1d_2}.
	\end{equation}
	In the case $d_2=1$, \eqref{eq:3.9} becomes a Ramanujan sum
	\begin{equation} \label{eq:3.13}
	\sum_{\substack{n\le N\\ (n,d_1)=1}}
	e_{d_1} (c_1\overline{n})
	\ll d_1^{1/2} \tau(d_1) + \frac{(c_1,d_1)N}{d_1};
	\end{equation}
	see, e.g., \cite[Lemma 6]{BFI1986} for a proof.
\end{lemma}

This next lemma is the Birch-Bombieri bound.

\begin{lemma}(\cite[Lemma 12]{Zhang2014} \label{lemma:12})
	Let
	\begin{equation}
	T(k;m_1,m_2;q)
	= 
	\sum_{\substack{\ell (q)\\ (\ell(\ell+k),q)=1}}
	\displaystyle\sideset{}{^*}\sum_{t_1 (q)}
	\displaystyle\sideset{}{^*}\sum_{t_2 (q)}	
	e_q\left(\overline{\ell} t_1 - \overline{(\ell+k)} t_2 + m_1\overline{t}_1 - m_2 \overline{t}_2 \right).
	\end{equation}
	Suppose that $q$ is square-free. Then we have	
	\begin{equation}
	T(k;m_1,m_2;q) \ll (k,q)^{1/2} q^{3/2} \tau(q).
	\end{equation}
\end{lemma}

\section{Proof of the main result Theorem \ref{thm:maintheoremtauk}}

We start with the proof of Theorem \ref{thm:maintheoremtauk} which is the longest of the four. We begin by making some preliminary reductions. Writing	\begin{equation} \label{eq:Q0}
	Q_0 = X^{\frac{1}{12(k+1)}}
\end{equation}
and
\begin{equation} \label{eq:D2}
	D_2 = \frac{X^{1/2}}{Q_0} = X^{\frac{1}{2}-\frac{1}{12(k+1)}},
\end{equation}
we first show that contributions coming from moduli $d\le D_2$, which we call small moduli, are acceptable. (See Remark \ref{remark:D2} below for a discussion on the dependency of $D_2$ on $k$.) The main ingredients in this first step are the multiplicative large sieve inequality \eqref{eq:largesieveineq} in conjunction with Lemma \ref{lemma:tinymoduli} to control contributions from primitive characters with very small moduli.

\begin{prop}\label{prop:mediummoduli}
	Let 
	\begin{equation}\label{eq:gamma}
	\tau_k=\tau_\ell * \tau_s
	\end{equation}
	with $k=\ell+s$, $\tau_\ell$ supported on $[M,2M)$, $\tau_s$ supported on $[N,2N)$, $M,N>X^{1/6(k+1)}$, and $MN=X$. For $D\le D_2$ we have
	\begin{equation}\label{eq:mediummoduli}
	\sum_{d\le D}\max_{(a,d)=1}
	|\Delta(\tau_k;X,d,a)|
	\ll 
	X^{1 - \frac{1}{12(k+2)}}.
	\end{equation}
\end{prop}
\begin{proof}
	As mentioned above, we estimate each $|\Delta(\tau_k;X,d,a)|$ directly using the large sieve inequality \eqref{eq:largesieveineq}. By Lemma \ref{lemma:nonprincipalsum}, we first reduce the task of estimating $|\Delta(\tau_k;X,d,a)|$ to a sum over nonprincipal characters, then, with Lemmas \ref{lemma:primitivesum} and \ref{lemma:tinymoduli} and the factorization
	\begin{equation}
	d=qr,
	\end{equation}
	we further reduce this sum to one involving only primitive characters, to which, we apply the large sieve inequality \eqref{eq:largesieveineq} to obtain \eqref{eq:mediummoduli}.
	
	By Lemma \ref{lemma:nonprincipalsum}, the left side of \eqref{eq:mediummoduli} is
	\begin{equation}
	\le \sum_{d\le D} \frac{1}{\varphi(d)} \sideset{}{'}\sum_{\chi(\text{mod }d)} \left| \sum_{n\le X} \tau_k(n)\chi(n) \right|.
	\end{equation}
	By the bound
	$\frac{1}{\varphi(d)}\ll \frac{\log\mathcal{L}}{d}$ and Lemma \ref{lemma:primitivesum}, this is
	\begin{align}
	&\ll \log\mathcal{L} \sum_{d\le D}\frac{1}{d} \sum_{\substack{d=qr\\ q>1}}
	\sideset{}{^*}\sum_{\ \chi^*(\text{mod } q)}
	\left| \sum_{\substack{n\le X\\ (n,r)=1}} \tau_k(n)\chi^*(n) \right|\\
	&= \log\mathcal{L} \sum_{r\le D}\frac{1}{r}
	\left(\sum_{1<q\le D/r} \frac{1}{q} \sideset{}{^*}\sum_{\chi(\text{mod } q)}\left|\sum_{\substack{n\le X\\ (n,r)=1}} \tau_k(n)\chi(n) \right| \right).
	\end{align}
	The sum $\sum_{r\le D}\frac{1}{r}$ contributes a factor of $\log D$. Thus, to show \eqref{eq:mediummoduli}, it suffices to show that, for each fixed $r\le D$,
	\begin{equation}\label{eq:sumoverprimitive}
	\sum_{1<q\le D/r} \frac{1}{q}\sideset{}{^*}\sum_{\chi(\text{mod } q)}\left|\sum_{\substack{n\le X\\ (n,r)=1}} \tau_k(n)\chi(n) \right|
	\ll 
	X^{1- 1/12(k+1)}.
	\end{equation}
	
	Fix an $r\le D$. Recall $Q_0$ given as in \eqref{eq:Q0}. We split the range of primitive conductors $q\in (1, D/r]$ into two, one over $q< Q_0$ and the other over $Q_0 \le q \le D/r$, with the intention of applying Lemma \ref{lemma:tinymoduli} to the former and large sieve inequality \eqref{eq:largesieveineq} to the latter.
	
	By Lemma \ref{lemma:tinymoduli}, we have, for $q\le Q_0$,
	\begin{equation}
		\sum_{n\le X} \tau_k(n) \chi(n)
		\ll X^{1-1/3(k+1)}.
	\end{equation}
	And since the number of characters with modulus less than $Q_0$ is at most $Q_0^2$, we get
	\begin{equation}
		\sum_{1<q< Q_0} \frac{1}{q}\ \sideset{}{^*}\sum_{\chi(\text{mod } q)}\left|\sum_{(n,r)=1} \tau_k(n)\chi(n) \right|
		\ll 
		Q_0^2 X^{1-1/3(k+1)}
		\sum_{1<q< D_0} \frac{1}{q}	
		\ll X^{1-1/6(k+2)}.	
	\end{equation}
	
	Thus, to prove \eqref{eq:sumoverprimitive} it is sufficient to show
	\begin{equation} \label{eq:97}
	\frac{1}{Q}\sum_{Q\le q\le 2Q}
	\sideset{}{^*}\sum_{\ \chi(\text{mod } q)}
	\left|\sum_{(n,r)=1} \tau_k(n)\chi(n) \right|
	\ll 
	X^{1- 1/12(k+1)}
	\end{equation}
	for any $Q_0 \le Q\le  D$. By \eqref{eq:gamma}, we have
	\begin{equation}\label{eq:crucialdecomposition}
	\sum_{(n,r)=1} \tau_k(n)\chi(n)
	=\left(\sum_{(m,r)=1} \tau_\ell(m)\chi(m) \right) 
	\left(\sum_{(n,r)=1} \tau_s(n)\chi(n) \right).
	\end{equation}
	By \eqref{eq:crucialdecomposition} and Cauchy's inequality, the left side of \eqref{eq:97} is
	\begin{equation}
	\le \frac{1}{Q}
	\left(\sum_{Q\le q\le 2Q} \sideset{}{^*}\sum_{\ \chi(\text{mod }q)} \left|\sum_{(m,r)=1} \tau_\ell(m)\chi(m) \right|^2 \right)^{1/2}
	\left(\sum_{Q\le q\le 2Q} \sideset{}{^*}\sum_{\ \chi(\text{mod }q)} \left|\sum_{(n,r)=1} \tau_s(n)\chi(n) \right|^2 \right)^{1/2}.
	\end{equation}
	By the large sieve inequality \eqref{eq:largesieveineq}, the above quantity is
	\begin{align}
	&\le \frac{1}{Q}\left( (Q^2+M) M^{2\epsilon}
	\right)^{1/2} 
	\left( (Q^2+N) N^{2\epsilon} \right)^{1/2}	=\frac{1}{Q}(Q^2+M)^{1/2}(Q^2+N)^{1/2}  X^{1/2+\epsilon}.
	\end{align}
	By the inequality $\sqrt{x+y}< \sqrt{x} + \sqrt{y}$, the above is bounded by
	\begin{align}
	&< \frac{1}{Q}(Q+\sqrt{M})(Q+\sqrt{N})
	X^{1/2+\epsilon}
	=\frac{1}{Q} (Q^2+Q\sqrt{M} + Q\sqrt{N} + X^{1/2})
	X^{1/2+\epsilon}\\
	\label{eq:almostthere}
	&= \left(Q+\sqrt{M} + \sqrt{N} +\frac{X^{1/2}}{Q} \right) 
	X^{1/2+\epsilon}.
	\end{align}
	Since $Q_0<Q<D_2$ and $N,M>X^{1/6(k+1)}$, we have
	\begin{equation}
	\begin{cases}
	Q<D_2=X^{1/2-1/12(k+1)},\\
	\sqrt{M} =\sqrt{\dfrac{X}{N}} < \sqrt{\dfrac{X}{X^{1/6(k+1)}}} 
	< X^{1/2-1/12(k+1)},\\
	\sqrt{N} =\sqrt{\dfrac{X}{M}} < \sqrt{\dfrac{X}{X^{1/6(k+1)}}} 
	< X^{1/2-1/12(k+1)},\\
	\dfrac{X^{1/2}}{Q} < \dfrac{X^{1/2}}{Q_0} = X^{1/2-1/12(k+1)}.
	\end{cases}
	\end{equation}
	Combining the above estimates, \eqref{eq:almostthere} is
	\begin{equation}
	\ll	
	X^{1-1/12(k+1) +\epsilon}
	\ll
	X^{1- 1/12(k+2)}.
	\end{equation}
	This leads to \eqref{eq:mediummoduli}.
\end{proof}

Thus, by Proposition \ref{prop:mediummoduli}, Theorem \ref{thm:maintheoremtauk} holds for $1\le d \le D_2$. From this, to prove \eqref{eq:tauk}, it suffices to show that
\begin{equation} \label{eq:tauk2}
	\sum_{\substack{d\in \mathcal{D}\\ D_2 < d < D_3}}
	|\Delta(\tau_k;X,d,a)|
	\ll X^{1-\theta_k}.
\end{equation}
The rest of this section is devoted to proving the estimate \eqref{eq:tauk2}.

\begin{remark} \label{remark:D2}
	The cutoff parameter $D_2$ introduced in \eqref{eq:D2} separating small and large moduli unfortunately has a dependence on $k$. We are unable to resolve this dependency without appealing to GRH or the Lindel\"of Hypothesis. This dependency on $k$ is the result of the convexity bound \eqref{eq:Lkbound} of $L(s,\chi)^k$ on the critical line which depends on $k$. On GRH, the bound on $L(1/2+it,\chi)^k$ can be made uniform in $k$; see Section \ref{section:proofOfUniformPowerSavings} below for more.
\end{remark}

\subsection{Combinatorial argument}

The goal of this subsection is to apply combinatorial arguments to reduce the proof of \eqref{eq:tauk2} to showing that
\begin{equation}\label{eq:6.2}
	\sum_{\substack{d\in \mathcal{D}\\ D_2 < d < D_3}}
	|\Delta(\gamma;X,d,a)| 
	\ll X^{1-\theta_k},
\end{equation}
for suitable $\gamma$. Combinatorial arguments amount to estimating $\sum_{\substack{n\le X\\ n\equiv a (\text{mod } d)}} \tau_k(n)$ by a sum of $\sum_{n \equiv a (d)} \gamma(n)$, where each $\gamma$ is of the form
\begin{equation}
\gamma=\beta_1*\beta_2*\dots * \beta_k,
\end{equation}
a convolution of simpler arithmetic functions $\beta_j$.

Following the fundamental work of Friedlander and Iwaniec in their treatment of the ternary divisor function $\tau_3(n)$ in \cite[Section 3]{FriedlanderIwaniec1985}, after decomposing the interval $[1,X]$ to $O(\mathcal{L}^B)$ dyadic intervals of the form $[N,2N)$, we perform a finner-than-dyadic subdivision of the interval $[N,2N)$ as follows. Let 
\begin{equation} \label{eq:rho}
	\rho=X^{-\varpi}.
\end{equation}
Let $R$ be the largest positive integer $r$ for which $(1+\rho)^r < 2x$. We have the following bound for $R$:
\begin{equation}
	R \le \frac{\log 2x}{\log (1+\rho)}
	\ll \rho^{-1} \log X.
\end{equation}
For $n\sim N$, we have $\tau_k(n) = T_1(n)$ where
\begin{equation} \label{eq:6.6}
	T_1(n)
	= \sum_{\mathcal{N} = (N_1,N_2,\cdots, N_k)} 
	\chi_{N_k}*\chi_{N_{k-1}}* \cdots
	* \chi_{N_1}.
\end{equation}
Here $N_1,N_2,\dots, N_k\ge 1$ run over the powers of $1+\rho$ satisfying
\begin{equation} \label{eq:6.4}
	[N_k \cdots N_1, (1+\rho)^k N_k \cdots N_1) \cap [N,2N) 
	\neq \emptyset.
\end{equation}
Let $T_2$ have the same expression as $T_1$ but with the constraint \eqref{eq:6.4} replaced by
\begin{equation} \label{eq:6.5}
	[N_k \cdots N_1, (1+\rho)^k N_k \cdots N_1)
	\subset [N,2N).
\end{equation}
Since $T_1-T_2$ is supported on $[(1+\rho)^{-k}N, (1+\rho)^k N]$ and $(T_1-T_2)(n) \ll \tau_k$, by Lemma \ref{lemma:8} we have
\begin{equation}
	\sum_{\substack{d\in \mathcal{D}\\ D_2 < d < D_3}}
	|\Delta(T_1 - T_2;X,d,a)| 
	\ll X^{1-\varpi^2}.
\end{equation}
Let $\gamma$ be of the form
\begin{equation}
	\gamma = \chi_{N_k}*\chi_{N_{k-1}}* \cdots.
	* \chi_{N_1}
\end{equation}
with $N_k, \dots, N_1$ satisfying \eqref{eq:6.5} and $N_k \le \cdots \le N_1$. Write $N_i=X^{\nu_i}$. We have
\begin{equation} \label{eq:7}
	0\le \nu_k \le \dots \le \nu_1,
\end{equation}
and
\begin{equation} \label{eq:8}
	0\le \nu_k+\dots + \nu_1 < 1- k\frac{\log \rho}{\mathcal{L}}.
\end{equation}
We deduce the proof of \eqref{eq:tauk2} from the following
\begin{prop} \label{prop:2}
	With the same notation as above, we have
	\begin{equation} \label{eq:6.2b}
		\sum_{\substack{d\in \mathcal{D}\\ D_2 < d < D_3}}
		|\Delta(\gamma;X,d,a)| 
		\ll X^{1-\varpi^{7/2}}.
	\end{equation}
\end{prop}
In view of \eqref{eq:7}, the proof of \eqref{eq:6.2b} is divided into three cases:

\bigskip
\textbf{Case (a):}
\begin{equation}
	\nu_1\ge \frac{5}{8}-8\varpi.
\end{equation}

\textbf{Case (b): }
\begin{equation} \label{eq:9}
	\nu_1< \frac{5}{8}-8\varpi
\end{equation}
and 
\begin{equation} \label{eq:10}
\sum_{i\in I} \nu_i
\notin
\left[\frac{3}{8}+8\varpi, \frac{5}{8} -8\varpi \right]
\end{equation}
for any subset $I$ of $\{1,2,\dots, k\}$.

\textbf{Case (c):} $$\nu_1< \frac{5}{8} -8\varpi$$ and there is a subset $I$ of $\{1,2,\dots,k \}$ such that
\begin{equation}
\sum_{i\in I} \nu_i
\in
\left[\frac{3}{8}+8\varpi, \frac{5}{8} -8\varpi \right].
\end{equation}

By Lemma \ref{lemma:sharpsiegelwalfiszbound}, all the choices of $\beta_i=\chi_{N_i}$ above satisfy the Siegel-Walfisz condition \eqref{eq:sharpsiegelwalfiszbound}. Noting that the sum in \eqref{eq:6.6} contains $R \ll \rho^{-1} \log X\ll X^{\varpi^{15/16}}$ terms, by the above discussion, we conclude that \eqref{eq:6.2b} implies \eqref{eq:tauk2}.

We start with the proof of Case (a), the simplest of the three.

\subsection{Proof of Case (a)}

This is the simplest case of the three. Let
\begin{equation}
	\beta = \beta_1,\quad N=N_1,
\end{equation}
and 
\begin{equation}
	\alpha = \beta_2*\cdots * \beta_k,
	\quad
	M= N_2\cdots N_k,
\end{equation}
so that $\gamma = \alpha*\beta$. Since $\nu_1\ge 5/8-8\varpi$ in this case, by Lemma \ref{lemma:sharpsiegelwalfiszbound} with $\kappa=X^{5/8-8\varpi}$, we have
\begin{equation}
	\Delta(\beta; X, d, a)
	\ll X^{5/16-4\varpi} N.
\end{equation}
By definition of $\Delta(\gamma;X,d,a)$ and bounding
\begin{equation}
	\sum_{m \sim M} \alpha(m)
	\ll \sum_{m \sim M} \tau_k(m)
	\ll M \mathcal{L}^B
\end{equation}
trivially, we get
\begin{equation}
	\Delta(\gamma;X,d,a)
	\le \sum_{(m,d)=1}
	\left| \sum_{m \sim M}
	\alpha(m)\right| 
	|\Delta(\beta;X,d, a\overline{m})|
	\ll M \mathcal{L}^B
	\max_{(a,d)=1} \Delta(\alpha;X,d,a)
	\ll X^{5/16-5\varpi}.
\end{equation}
Thus,
\begin{equation}
	\sum_{d\le D_3} \Delta(\gamma;X,d,a)
	\ll X^{5/16-5\varpi} D_3
	\le X^{13/16 - 3\varpi}.
\end{equation}
This leads to \eqref{eq:6.2}.

\subsection{Proof of Case (b)} This case corresponds to the Type III estimate in \cite[\S\S 13, 14]{Zhang2014}, which we will follow closely. The main tool we need is the Birch-Bombieri bound from Lemma \ref{lemma:12}. We will reduce $\Delta(\gamma; X, d, a)$ into an exponential sum of that form; see \eqref{eq:1214}. Write $\alpha=\beta_4*\cdots* \beta_k$ so that
\begin{equation}
\gamma=\alpha*\beta_1*\beta_2*\beta_3.
\end{equation}
Let
\begin{equation}
M= N_4\cdots N_k.
\end{equation}
Note that $\alpha*\beta_1$ is supported on $[MN_1,2MN_1)$. We have the following lemma.
\begin{lemma} \label{lemma:combinatorialLemma}
	Suppose that
	\begin{equation} \label{eq:comblemmaeq1}
	\nu_1 < \frac{5}{8} - 8\varpi
	\end{equation}
	and, for any subset $I\subseteq \{1,2,\cdots, k\}$,
	\begin{equation} \label{eq:comblemmaeq2}
	\sum_{i\in I} \nu_i 
	\notin \left[
	\frac{3}{8}+8\varpi, \frac{5}{8} - 8\varpi
	\right].
	\end{equation}
	Then we have
	\begin{equation}
	\nu_3+ \nu_2 > \frac{5}{8} - 8\varpi.
	\end{equation}
\end{lemma}
\begin{proof}
	By virtue of \eqref{eq:comblemmaeq2} with $I=\{2,3\}$, it suffices to show that
	\begin{equation}
	\nu_3+ \nu_2 > \frac{3}{8} + 8\varpi.
	\end{equation}
	By \eqref{eq:comblemmaeq1} and \eqref{eq:comblemmaeq2} we have
	\begin{equation}
	\nu_1 < \frac{3}{8} + 8\varpi.
	\end{equation}
	By \eqref{eq:7},
	\begin{equation}
	\nu_2 \le \nu_1 < \frac{3}{8} + 8\varpi.
	\end{equation}
	By the first inequality in \eqref{eq:8},
	\begin{equation}
	\nu_2 + \cdots + \nu_k
	= (\nu_k + \cdots + \nu_1) - \nu_1
	> 1- \left(\frac{3}{8} + 8\varpi\right)
	= \frac{5}{8} - 8\varpi.
	\end{equation}
	Hence, by \eqref{eq:8}, there is an $3\le \ell\le k$ such that
	\begin{equation}
	\nu_2 + \cdots + \nu_{\ell-1} 
	< \frac{3}{8} + 8\varpi
	\end{equation}
	and
	\begin{equation}
	\nu_2 + \cdots + \nu_{\ell} 
	> \frac{5}{8} - 8\varpi,
	\end{equation}
	so that
	\begin{equation}
	\nu_{\ell} > \frac{1}{4} - 16\varpi
	\end{equation}
	Thus, by \eqref{eq:7} and \eqref{eq:comblemmaeq2}, we conclude
	\begin{equation}
	\nu_3 + \nu_2 \ge 2\nu_{\ell} > \frac{1}{2} - 32\varpi
	> \frac{3}{8} + 8\varpi.
	\end{equation}
\end{proof}

We have
\begin{equation} \label{eq:13.2}
N_1 \ge N_2 \ge \left(\frac{X}{MN_1}\right)^{1/2}
\ge X^{5/16-4\varpi},
\end{equation}
and
\begin{equation} \label{eq:13.3}
N_3 \ge \frac{X}{MN_1 N_2}
\ge X^{1/4-16\varpi}.
\end{equation}
By Lemma \ref{lemma:combinatorialLemma} we have
\begin{equation} \label{eq:12}
\nu_3 + \nu_2 > \frac{5}{8} - 8\varpi.
\end{equation}
Hence,
\begin{equation} \label{eq:12b}
MN_1 \ll \frac{X}{N_2N_3} \ll X^{3/8+8\varpi}.
\end{equation}

Let $f$ be as in Lemma \ref{lemma:7} with $\eta^*=\eta$, $\epsilon=\varpi$, and with $N_1$ in place of $M$. Note that the function $\beta_1-f$ is supported on $[N_1^-,N_1]\cup [\eta N_1, \eta N_1^+]$ with $N_1^\pm = (1\pm N_1^{-\varpi}) N_1$. Letting
\begin{equation}
\gamma^* = \alpha*\beta_2*\beta_3* f,
\end{equation}
we have
\begin{align}
\sum_{(n,d)=1} (\gamma-\gamma^*)(n)
&\ll \sum_{(mn,d)=1} \alpha*\beta_2*\beta_3(n) \beta_1(m) - \alpha*\beta_2*\beta_3(n) f(m)\\
&= \sum_{(n,d)=1} \alpha*\beta_2*\beta_3(n)
\sum_{(m,d)=1} (\beta_1(m) - f(m) )\\
&\ll \frac{\varphi(d)}{d} MN_2N_3M^\epsilon
(N_1^{1-\varpi} + \eta N_1^{1-\varpi})\\
&\ll \frac{\varphi(d)}{d} X^{1-\varpi/2}
\end{align}
and
\begin{align}
\sum_{n\equiv a(d)} (\gamma-\gamma^*)(n)
&\ll \sum_{\substack{N_1^- \le q \le N_1\\ (q,d)=1}}
\sum_{\substack{1\le \ell < 3x/q\\ \ell q \equiv a (d)}} \tau_k(\ell)
+ \sum_{\substack{\eta N_1 \le q \le \eta N_1^+\\ (q,d)=1}}
\sum_{\substack{1\le \ell < 3x/q\\ \ell q \equiv a (d)}} \tau_k(\ell)\\
&\ll \frac{\varphi(d)}{d^2} MN_2N_3 N_1^{1-\varpi} \mathcal{L}^B
\ll \frac{X^{1-\varpi/2}}{d}.
\end{align}
It therefore suffices to prove \eqref{eq:6.2} with $\gamma$ replaced by $\gamma^*$. We shall prove the  bound
\begin{equation} \label{eq:13.4}
\Delta(\gamma^*;d,c)
\ll \frac{X^{1-\varpi/3}}{d}.
\end{equation}

We remove the condition $(n,d)=1$ by means of M\"obius inversion formula
\begin{equation} \label{eq:mobiusInversion}
\sum_{\delta|(n,d)} \mu(\delta)
= \begin{cases}
1, & \text{if } (n,d)=1,\\
0, & \text{otherwise}.                         
\end{cases}
\end{equation}
By \eqref{eq:mobiusInversion} we have
\begin{equation}
\sum_{(n,d)=1} f(n)
= \sum_{n} f(n)
\sum_{\substack{\delta|n\\ \delta| d}} \mu(\delta)
= \sum_{\substack{\delta|d}} \mu(\delta)
\sum_{n \equiv0 (\delta)} f(n).
\end{equation}
For $\delta> N^{1-2\epsilon}$, the inner sum on the right side is
\begin{equation}
\sum_{n\equiv 0 (\delta)} f(n)
\ll \frac{N}{N^{1-2\epsilon}}
\ll N^{2\epsilon}.
\end{equation}
This yields
\begin{equation} 
\sum_{(n,d)=1} f(n)
= \sum_{\substack{\delta|d\\ \delta \le N^{1-2\epsilon}}} \mu(\delta)
\sum_{n \equiv0 (\delta)} f(n) + O(N^\epsilon)
= \sum_{\substack{\delta|d\\ \delta \le N^{1-2\epsilon}}} \frac{\mu(\delta)}{\delta} \hat{f}(0) +O(N^\epsilon).
\end{equation}
Since
\begin{equation}
\sum_{\delta|d} \frac{\mu(\delta)}{\delta}
= \frac{\varphi(d)}{d},
\end{equation}
the above becomes
\begin{equation} \label{eq:sharpsiegelwalfiszbound5}
\sum_{(n,d)=1} f(n)
= \frac{\varphi(d)}{d} \hat{f}(0) +O(N^\epsilon).
\end{equation}
By \eqref{eq:13.2}, this yields
\begin{equation}
\frac{1}{\varphi(d)}
\sum_{(n,d)=1} \gamma^*(n)
= \frac{\hat{f}(0)}{d}
\sum_{(m,d)=1}
\sum_{\substack{n_3 \simeq N_3\\ (n_3,d)=1}}
\sum_{\substack{n_2 \simeq N_2\\ (n_2,d)=1}}
\alpha(m)
+ O(d^{-1} X^{3/4}).
\end{equation}
Here $n\simeq N$ stands for $N\le n\le \eta N$. On the other hand, we have
\begin{equation}
\sum_{n\equiv a (d)} \gamma^*(n)
= \sum_{(m,d)=1}
\sum_{\substack{n_3 \simeq N_3\\ (n_3,d)=1}}
\sum_{\substack{n_2 \simeq N_2\\ (n_2,d)=1}}
\alpha(m)
\sum_{mn_3n_2n_1 \equiv a(d)} f(n_1).
\end{equation}
The innermost sum is, by Lemma \ref{lemma:7}, equal to
\begin{equation}
\frac{1}{d} \sum_{|h|< H^*} \hat{f}(h/d)
e_d(-ah\overline{mn_3n_2}) 
+ O(d^{-1})
\end{equation}
where
\begin{equation}
H^* = dN_1^{-1+2\epsilon}.
\end{equation}
It follows that the left side of \eqref{eq:13.4} is
\begin{equation}
= \frac{1}{d}
\sum_{\substack{m\simeq M\\ (m,d)=1}}
\sum_{\substack{n_3\simeq N_3\\ (n_3,d)=1}}
\sum_{\substack{n_2\simeq N_2\\ (n_2,d)=1}}
\alpha(m)
\sum_{1\le |h| < H^*} \hat{f}(h/d) e_d(-ah\overline{mn_3n_2})
+ O(d^{-1} X^{3/4}).
\end{equation}
Hence the proof of \eqref{eq:13.4} is reduced to showing that
\begin{equation} \label{eq:13.5}
\sum_{1\le h < H^*}
\sum_{\substack{n_3\simeq N_3\\ (n_3,d)=1}}
\sum_{\substack{n_2\simeq N_2\\ (n_2,d)=1}}
\hat{f}(h/d) e_d(ah\overline{n_3n_2})
\ll X^{1-\varpi/2+2\epsilon} M^{-1}
\end{equation}
for any $a$ with $(a,d)=1$. Substituting $d_1=d/(h,d)$ and applying M\"obius inversion, the left side of \eqref{eq:13.5} can be written as
\begin{align}
\sum_{d_1|d}
&\sum_{\substack{1\le h < H\\ (h,d_1)=1}}
\sum_{\substack{n_3\simeq N_3\\ (n_3,d)=1}}
\sum_{\substack{n_2\simeq N_2\\ (n_2,d)=1}}
\hat{f}(h/d_1) e_{d_1}(ah \overline{n_3n_2})\\
= &\sum_{d_1 d_2=d}\ \sum_{b_3|d_2}\ \sum_{b_2|d_2}
\mu(b_3) \mu(b_2)
\sum_{\substack{1\le h< H\\ (h,d_1)=1}}\
\sum_{\substack{n_3 \simeq N_3/b_3\\ (n_3,d_1)=1}}\
\sum_{\substack{n_2 \simeq N_2/b_2\\ (n_2,d_1)=1}}
\hat{f}(h/d_1)
e_{d_1} (ah \overline{b_3b_2 n_3n_2}),
\end{align}
where
\begin{equation} \label{eq:13.6}
H = d_1 N_1^{-1+2\epsilon}.
\end{equation}
It therefore suffices to show that
\begin{equation} \label{eq:13.7}
\sum_{\substack{1\le h<H\\ (h,d_1)=1}}\
\sum_{\substack{n_3 \simeq N_3'\\ (n_3,d_1)=1}}\
\sum_{\substack{n_2 \simeq N_2'\\ (n_2,d_1)=1}}\
\hat{f}(h/d_1)
e_{d_1}(bh\overline{n_3n_2})
\ll X^{1-\varpi/2+\epsilon} M^{-1}
\end{equation}
for any $d_1,b,N_3'$, and $N_2'$ satisfying
\begin{equation} \label{eq:13.8}
d_1|d,\quad
(b,d_1)=1, \quad
\frac{d_1 N_3}{d} \le N_3'\le N_3,\quad
\frac{d_1N_2}{d} \le N_2'\le N_2,
\end{equation}
which are henceforth assumed. By the lower bound \eqref{eq:13.2} we have
\begin{equation} \label{eq:13.9}
H \ll X^{3/16+6\varpi + \epsilon}.
\end{equation}
By \eqref{eq:13.6}, the left side of \eqref{eq:13.7} is void if $d_1\le N_1^{1-2\epsilon}$, so we may assume that $d_1>N_1^{1-2\epsilon}$. By the trivial bound
\begin{equation} \label{eq:13.10}
\hat{f}(z) \ll N_1
\end{equation}
and the bound \eqref{eq:3.13}, we find that the left side of \eqref{eq:13.7} is
\begin{equation}
\ll HN_3N_1(d_1^{1/2+\epsilon} + d_1^{-1} N_2)
\ll d_1^{3/2+\epsilon} N_1^{2\epsilon} N_3.
\end{equation}
If $d_1\le X^{5/12-6\varpi}$, then the right side of the above is $\ll X^{1-\varpi+3\epsilon}$ by \eqref{eq:12b} and \eqref{eq:2.13}, and this leads to \eqref{eq:13.7}. Thus we may further assume that
\begin{equation} \label{eq:13.11}
d_1 > X^{5/12-6\varpi}.
\end{equation}

We appeal to the Weyl shift and the factorization $d_2=rq$. By Lemma \ref{lemma:4}, with $d_1$ in place of $d$, we can choose a factor $r$ of $d_1$ such that
\begin{equation} \label{eq:13.12}
X^{44\varpi} < r < X^{45\varpi}.
\end{equation}
Write
\begin{equation}
\mathcal{N}(d_1,k)
= \sum_{\substack{1\le h< H\\ (h,d_1)=1}}\
\sum_{\substack{n_3\simeq N_3'\\ (n_3,d_1)=1}}\
\sum_{\substack{n_2\simeq N_2'\\ (n_2+hkr,d_1)=1}}
\hat{f}(h/d_1)
e_{d_1}(bh\overline{(n_2+hkr)n_3}),
\end{equation}
so that $\mathcal{N}(d_1,0)$ corresponds to the left side of \eqref{eq:13.7}. Assume $k>0$. We have
\begin{equation} \label{eq:13.13}
\mathcal{N}(d_1,k) - \mathcal{N}(d_1,0)
= \mathcal{Q}_1(d_1,k) - \mathcal{Q}_2(d_1,k),
\end{equation}
where
\begin{equation}
\mathcal{Q}_i(d_1,k)
= \sum_{\substack{1\le h < H\\ (h,d_1)=1}}\
\sum_{\substack{n_3 \simeq N_3'\\ (n_3,d_1)=1}}\
\sum_{\substack{\ell \in \mathcal{I}_i(h)\\ (\ell,d_1)=1}}
\hat{f}(h/d_1)
e_{d_1}(bh\overline{\ell n_3}),\quad
i=1,2,
\end{equation}
with
\begin{equation}
\mathcal{I}_1(h) = [\eta N_2', \eta N_2'+hkr),\quad
\mathcal{I}_2(h) = [N_2', N_2'+hkr).
\end{equation}

By M\"obius inversion, we have
\begin{equation}
\mathcal{Q}_i(d_1,k)
= \sum_{st=d_1} \mu(s)
\sum_{1\le h< H/s}\
\sum_{\substack{n_3 \simeq N_3'\\ (n_3,d_1)=1}}\
\sum_{\substack{\ell \in \mathcal{I}_i(sh)\\ (\ell,d_1)=1}}
\hat{f}(h/t) e_t(bh \overline{\ell n_3}).
\end{equation}
The $h$ sum is empty unless $s<H$. Since $H^2 = o(d_1)$ by \eqref{eq:13.9} and \eqref{eq:13.11}, it follows, by changing the order of summation, that
\begin{equation}
|\mathcal{Q}_i(d_1,k)|
\le \sum_{\substack{st=d_1\\ t>H}}\
\sum_{\substack{n_3 \simeq N_3'\\ (n_3,d_1)=1}}\
\sum_{\substack{\ell \in \mathcal{I}_i(H)\\ (\ell,d_1)=1}}
\left|\sum_{h\in J_i(s,\ell)}
\hat{f}(h/t) e_t(bh\overline{\ell n_3}) \right|,
\end{equation}
where $J_i(s,\ell)$ is some interval of length at most $H$ depending on $s$ and $\ell$. By integration by parts, we have
\begin{equation}
\frac{d}{dz} \hat{f}(z)
\ll \min\{N_1^2, |z|^{-2}N_1^\epsilon \},
\end{equation}
and by partial summation and \eqref{eq:13.10} we obtain
\begin{equation}
\sum_{h\in J_i(s,\ell)}
\hat{f}(h/t) e_t(bh\overline{\ell n_3})
\ll N_1^{1+\epsilon}
\min\{H, \|b\overline{\ell n_3}/t \|^{-1} \}.
\end{equation}
Thus,
\begin{equation}
\mathcal{Q}_i(d_1,k)
\ll N_1^{1+\epsilon}
\sum_{\substack{t| d_1\\ t>H}}\
\sum_{\substack{\ell \in \mathcal{I}_i(H)\\ (\ell,d_1)=1}}\
\sum_{\substack{n_3< 2N_3\\ (n_3,d_1)=1}}
\min\{H,\|b\overline{\ell n_3}/t \|^{-1} \}.
\end{equation}
Since $H=o(N_3)$ by \eqref{eq:13.3} and \eqref{eq:13.9}, the inner most sum is $\ll N_3^{1+\epsilon}$, by Lemma \ref{lemma:9}. By \eqref{eq:13.6}, this leads to
\begin{equation} \label{eq:13.14}
\mathcal{Q}_i(d_1,k) \ll d_1^{1+\epsilon} kr N_3.
\end{equation}

We now introduce the parameter
\begin{equation} \label{eq:13.15}
K = [X^{-1/2-48\varpi}N_1N_2],
\end{equation}
which is $\gg X^{1/8-56\varpi}$ by \eqref{eq:13.2}. By the second inequality in \eqref{eq:13.12}, the right side of \eqref{eq:13.14} is $\ll X^{1-\varpi+\epsilon}M^{-1}$ if $k< 2K$. Hence, by \eqref{eq:13.13}, the proof of \eqref{eq:13.7} is reduced to showing that
\begin{equation} \label{eq:13.16}
\frac{1}{K} \sum_{k\sim K} \mathcal{N}(d_1,k)
\ll X^{1-\varpi/2+\epsilon} M^{-1}.
\end{equation}

We now prove \eqref{eq:13.16}. By the relation
\begin{equation}
h \overline{(n_2+hkr)}
\equiv \overline{\ell + kr}\ (\textrm{mod } d_1)
\end{equation}
for $(h,d_1) = (n_2+hkr,d_1)=1$, where $\ell\equiv \overline{h} n_2\ (\textrm{mod } d_1)$, we may rewrite $\mathcal{N}(d_1,k)$ as
\begin{equation}
\mathcal{N}(d_1,k)
= \sum_{\substack{\ell\ (\textrm{mod } d_1)\\ (\ell + kr,d_1)=1}}
\nu(\ell;d_1)
\sum_{\substack{n_3\simeq N_3'\\ (n_3,d_1)=1}}
e_{d_1}(b\overline{(\ell+kr)n_3}),
\end{equation}
where
\begin{equation}
\nu(\ell; d_1) = 
\displaystyle\sideset{}{'}\sum_{\overline{h}n_2 \equiv \ell (d_1)}
\hat{f}(h/d_1).
\end{equation}
Here $\displaystyle\sideset{}{'}\sum$ is the restriction to $1\le h <H$, $(h,d_1)=1$, and $n_2\simeq N_2'$. It follows by Cauchy's inequality that
\begin{equation} \label{eq:14.1}
\left|\sum_{k\sim K} \mathcal{N}(d_1,k) \right|^2
\le P_1P_2,
\end{equation}
where
\begin{equation}
P_1 = \sum_{\ell\ (\textrm{mod } d_1)}
|\nu(\ell; d_1)|^2
\quad \text{and} \quad
P_2 = \sum_{\ell\ (\textrm{mod } d_1)}
\left|\sum_{\substack{k\sim K\\ (\ell+kr,d_1)=1}}
\sum_{\substack{n\simeq N_3'\\ (n,d_1)=1}}
e_{d_1}(b\overline{(\ell+kr)n})
\right|^2.
\end{equation}
By \eqref{eq:13.10} we have
\begin{equation}
P_1 \ll N_1^2 N_4,
\end{equation}
where
\begin{align}
	N_4 &= \#\{(h_1,h_2; n_1,n_2): h_2n_1 \equiv h_1n_2\ (\textrm{mod } d_1), 1\le h_i <H, n_i\simeq N_2' \}\\
	&\ll \sum_{\ell\ (\textrm{mod } d_1)}
\left(\sum_{\substack{1\le m< 2HN_2\\ m\equiv \ell (d_1)}} \tau(m) \right)^2.
\end{align}
Since $HN_2 \ll d_1^{1+\epsilon}$ by \eqref{eq:13.6}, we get
\begin{equation} \label{eq:14.2}
P_1 \ll d_1^{1+\epsilon} N_1^2.
\end{equation}
We claim that
\begin{equation} \label{eq:14.3}
P_2 \ll d_1 X^{3/16+52\varpi +\epsilon} K^2.
\end{equation}
With this, the estimate \eqref{eq:13.16} follows from \eqref{eq:14.1}, \eqref{eq:14.2}, and \eqref{eq:14.3} immediately since
\begin{equation}
N_1 \le X^{3/8 + 8\varpi} M^{-1},\quad
d_1 < X^{1/2+2\varpi},
\end{equation}
and
\begin{equation}\label{eq:conditionvarpi}
	\frac{31}{32} + 36\varpi = 1-\frac{\varpi}{2}.
\end{equation}

It remains to prove \eqref{eq:14.3}. Write $d_1=rq$. Note that
\begin{equation} \label{eq:14.4}
\frac{N_3'}{r} \gg X^{1/6-69\varpi}
\end{equation}
by \eqref{eq:13.8}, \eqref{eq:13.11}, \eqref{eq:13.3}, and the second inequality in \eqref{eq:13.12}. Since
\begin{equation}
\sum_{\substack{n\simeq N_3'\\ (n,d_1)=1}}
e_{d_1}(b\overline{(\ell+kr)n})
= \sum_{\substack{0\le s<r\\ (s,r)=1}}\
\sum_{\substack{n\simeq N_3'/r\\ (nr+s,q)=1}}
e_{d_1}(b\overline{(\ell + kr)(nr+s)})
+ O(r),
\end{equation}
we get
\begin{equation}
\sum_{\substack{k\sim K\\ (\ell +kr, d_1)=1}}\
\sum_{\substack{n\simeq N_3'\\(n,d_1)=1}}
e_{d_1}(b\overline{(\ell+kr)n})
= U(\ell) + O(Kr),
\end{equation}
where
\begin{equation}
U(\ell) = 
\sum_{\substack{0\le s<r\\ (s,r)=1}}\
\sum_{\substack{k\sim K\\ (\ell+kr,d_1)=1}}\
\sum_{\substack{n\simeq N_3'/r\\ (nr+s,q)=1}}
e_{d_1}(b\overline{(\ell+kr)(rn+s)}).
\end{equation}
Hence,
\begin{equation} \label{eq:14.5}
P_2 \ll \sum_{\ell\ (\textrm{mod } d_1)}
|U(\ell)|^2 + d_1(Kr)^2.
\end{equation}
By the second inequality in \eqref{eq:13.12}, the second term on the right side of the above is admissible for \eqref{eq:14.3}. On the other hand, we have
\begin{equation} \label{eq:14.6}
\sum_{\ell\ (\textrm{mod } d_1)}
|U(\ell)|^2
= \sum_{k_1\sim K}
\sum_{k_2 \sim K}
\sum_{\substack{0\le s_1<r\\ (s_1,r)=1}}
\sum_{\substack{0\le s_2<r\\ (s_2,r)=1}}
V(k_2-k_1; s_1,s_2),
\end{equation}
where
\begin{equation} \label{eq:153}
V(k;s_1,s_2) = 
\sum_{\substack{n_1\simeq N_3'/r\\ (n_1r+s_1,q)=1}}
\sum_{\substack{n_2\simeq N_3'/r\\ (n_2r+s_2,q)=1}}
\displaystyle\sideset{}{'}\sum_{\ell\ (\textrm{mod } d_1)}
e_{d_1}(b\overline{\ell (n_1r+s_1)}- b\overline{(\ell +kr)(n_2r+s_2)}).
\end{equation}
Here $\displaystyle\sideset{}{'}\sum$ is the restriction to $(\ell,d_1)=(\ell+kr,d_1)=1$. Note that if $\ell \equiv \ell_1 r + \ell_2 q\ (\textrm{mod } d_1)$, then the condition $(\ell(\ell+kr),d_1)=1$ is equivalent to $(\ell_1(\ell_1+k),q) = (\ell_2,r)=1$. In this case, by the relation
\begin{equation}
\frac{1}{d_1} \equiv \frac{\overline{r}}{q} + \frac{\overline{q}}{r}\ (\textrm{mod }1),
\end{equation}
we have
\begin{align}
&\frac{\overline{\ell(n_1r+s_1)}- \overline{(\ell+kr)(n_2r+s_2)}}{d_1}\\
&\equiv \frac{\overline{r^2\ell_1(n_1r+s_1)} - \overline{r^2(\ell_1+k)(n_2r+s_2)}}{q}
+ \frac{\overline{q^2s_1s_2\ell_2}(s_2-s_1)}{r}\
(\textrm{mod }1).
\end{align}
Thus, by the Chinese remainder theorem, the innermost sum in \eqref{eq:153} is equal to
\begin{equation}
C_r(s_2-s_1)
\sum_{\substack{\ell\ (\textrm{mod } q)\\ (\ell(\ell+k),q)=1}}
e_q(b\overline{r^2\ell(n_1r+s_1)} - b\overline{r^2(\ell+k)(n_2r+s_2)}),
\end{equation}
and, thus,
\begin{equation} \label{eq:14.7}
V(k;s_1,s_2)
= W(k;s_1,s_2) C_r(s_2-s_1),
\end{equation}
where
\begin{align}
W(k,s_1,s_2)
= \sum_{\substack{n_1\simeq N_3'/r\\ (n_1r+s_1,q)=1}}
\sum_{\substack{n_2\simeq N_3'/r\\ (n_2r+s_2,q)=1}}
\displaystyle\sideset{}{'}\sum_{\ell\ (\textrm{mod } q)}
e_q(b\overline{r^2\ell(n_1r+s_1)} - b\overline{r^2(\ell+k)(n_2r+s_2)}).
\end{align}
Here $\displaystyle\sideset{}{'}\sum$ is the restriction to $(\ell(\ell+k),q)=1$.

We first estimate the contribution from terms with $k_1=k_2$ on the right side of \eqref{eq:14.6} as follows. For $(n_1r+s_1,q)=(n_2r+s_2,q)=1$, we have
\begin{equation}
\displaystyle\sideset{}{^*}\sum_{\ell\ (\textrm{mod } q)}
e_q((b\overline{r^2\ell(n_1r+s_1)} - b\overline{r^2\ell(n_2r+s_2)} )
= C_q((n_1-n_2)r + s_1-s_2).
\end{equation}
Since $N_3'\ll X^{1/3}$, by \eqref{eq:13.11} and the second inequality in \eqref{eq:13.12}, we have
\begin{equation} \label{eq:14.8}
\frac{N_3'}{d_1} \ll X^{-1/12+6\varpi}
\ll r^{-1}.
\end{equation}
This implies $N_3'/r = o(q)$, giving
\begin{equation}
\sum_{n\simeq N_3'/r}
|C_q(nr+m)|
\ll q^{1+\epsilon}
\end{equation}
for any $m$. Thus
\begin{equation}
W(0;s_1,s_2) \ll q^{1+\epsilon} r^{-1} N_3'.
\end{equation}
Substituting the above into \eqref{eq:14.7} and using the simple estimate
\begin{equation}
\sum_{0\le s_1<r} \sum_{0\le s_2<r}
|C_r(s_2-s_1)|
\ll r^{2+\epsilon},
\end{equation}
we get that
\begin{equation}
\sum_{\substack{0\le s_1<r\\ (s_1,r)=1}}\
\sum_{\substack{0\le s_2<r\\ (s_2,r)=1}}
V(0;s_1,s_2)
\ll d_1^{1+\epsilon} N_3.
\end{equation}
Hence the contribution from the terms with $k_1=k_2$ on the right side of \eqref{eq:14.6} is $\ll d_1^{1+\epsilon}KN_3$; this reduces to showing that
\begin{equation} \label{eq:14.9}
\sum_{k_1\sim K}\
\sum_{\substack{k_2\sim K\\ k_2\neq k_1}}\
\sum_{\substack{0\le s_1 < r\\ (s_1,r)=1}}\
\sum_{\substack{0\le s_2 < r\\ (s_2,r)=1}}\
V(k_2-k_1;s_1,s_2)
\ll d_1 X^{3/16+52\varpi +\epsilon} K^2.
\end{equation}

By \eqref{eq:14.4} and \eqref{eq:14.8}, letting
\begin{equation}
n' = \min\{n: n\simeq N_3'/r \},\quad
n'' = \max\{n: n\simeq N_3'/r \},
\end{equation}
we may rewrite $W(k;s_1,s_2)$ as
\begin{align}
\sum_{\substack{n_1\le q\\ (n_1r+s_1,q)=1}}\
\sum_{\substack{n_2\le q\\ (n_2r+s_2,q)=1}}\
\displaystyle\sideset{}{'}\sum_{\ell\ (\textrm{mod } q)}
F\left(\frac{n_1}{q}\right)
F\left(\frac{n_2}{q}\right)
e_q\left(b\overline{r^2\ell(n_1r+s_1)} - b\overline{r^2(\ell+k)(n_2r+s_2)} \right),
\end{align}
where $F(y)$ is a function of $C^2[0,1]$ class such that
\begin{equation}
0\le F(y) \le 1,
\end{equation}
\begin{equation}
F(y) = 
\begin{cases}
1,&  \text{if}\quad
y \in \left[
\frac{n'}{q}, \frac{n''}{q} \right],\\
0,& \text{if} \quad
y\notin \left[\frac{n'}{q}-\frac{1}{2q}, \frac{n''}{q} + \frac{1}{2q} \right],
\end{cases}
\end{equation}
for which the Fourier coefficient
\begin{equation}
\kappa(m) = \int_0^1 F(y) e(-my) dy
\end{equation}
satisfies
\begin{equation} \label{eq:14.10}
\kappa(m) \ll \kappa^*(m)
:= \min\left\{\frac{1}{r},\frac{1}{|m|}, \frac{q}{m^2} \right\}.
\end{equation}
Here we have used \eqref{eq:14.8}. Fourier expand $F(y)$ we obtain
\begin{equation} \label{eq:14.11}
W(k;s_1,s_2)
= \sum_{m_1=-\infty}^\infty\ \sum_{m_2=-\infty}^\infty
\kappa(m_1) \kappa(m_2)
Y(k;m_1,m_2;s_1,s_2),
\end{equation}
where
\begin{equation}
Y(k;m_1,m_2;s_1,s_2)
= \sum_{\substack{n_1\le q\\ (n_1r+s_1,q)=1}}\
\sum_{\substack{n_2\le q\\ (n_2r+s_2,q)=1}}\
\displaystyle\sideset{}{'}\sum_{\ell\ (\textrm{mod } q)}
e_q(\delta(\ell,k;m_1,m_2;n_1,n_2;s_1,s_2)),
\end{equation}
with
\begin{align}
\delta(\ell,k;m_1,m_2;n_1,n_2;s_1,s_2)
=\ b\overline{r^2\ell(n_1r+s_1)} - b\overline{r^2(\ell+k)(n_2r+s_2)}
+ m_1n_1 + m_2n_2.
\end{align}
Moreover, if $n_jr + s_j \equiv t_j\ (\textrm{mod }q)$, then $n_j \equiv \overline{r} (t_j-s_j)\ (\textrm{mod } q)$ so that
\begin{equation}
m_1n_1 + m_2n_2
\equiv \overline{r} (m_1t_1 + m_2t_2)
- \overline{r}(m_1s_1+ m_2s_2)\ 
(\textrm{mod } q).
\end{equation}
Hence, on substituting $n_jr + s_j = t_j$, we may write $Y(k;m_1,m_2;s_1,s_2)$ as
\begin{equation} \label{eq:14.12}
Y(k;m_1,m_2;s_1,s_2) = 
Z(k;m_1,m_2) e_q\left(-\overline{r}(m_1s_1+m_2s_2) \right),
\end{equation}
where
\begin{equation}
Z(k;m_1,m_2)
= \displaystyle\sideset{}{^*}\sum_{t_1\ (\textrm{mod } q)}\
\displaystyle\sideset{}{^*}\sum_{t_2\ (\textrm{mod } q)}\
\displaystyle\sideset{}{'}\sum_{\ell\ (\textrm{mod } q)}
e_q(b\overline{r^2\ell t_1} - b\overline{(r^2(\ell+k)t_2)} + \overline{r}(m_1t_1 + m_2t_2) ).
\end{equation}
By \eqref{eq:14.7}, \eqref{eq:14.11}, and \eqref{eq:14.12} we get
\begin{align} \label{eq:14.13}
\sum_{\substack{0\le s_1<r\\ (s_1,r)=1}}\
\sum_{\substack{0\le s_2<r\\ (s_2,r)=1}}
V(k;s_1,s_2)
=
\sum_{m_1=-\infty}^\infty\ \sum_{m_2=-\infty}^\infty
\kappa(m_1) \kappa(m_2)
Z(k;m_1,m_2) J(m_1,m_2),
\end{align}
where
\begin{equation}
J(m_1,m_2)
= \sum_{\substack{0\le s_1<r\\ (s_1,r)=1}}\
\sum_{\substack{0\le s_2<r\\ (s_2,r)=1}}
e_q(-\overline{r} (m_1s_1+m_2s_2) )
C_r(s_2-s_1).
\end{equation}
We now appeal to Lemma \ref{lemma:12}.

By simple substitution we have
\begin{equation}
Z(k;m_1,m_2) = T(k,bm_1\overline{r}^3, -bm_2\overline{r}^3; q),
\end{equation}
so Lemma \ref{lemma:12} gives
\begin{equation} \label{eq:1214}
Z(k;m_1,m_2) \ll (k,q)^{1/2} q^{3/2+\epsilon},
\end{equation}
the right side does not depend on $m_1$ and $m_2$. We claim the following estimate
\begin{equation} \label{eq:14.14}
\sum_{m_1=-\infty}^\infty\ \sum_{m_2=-\infty}^\infty
\kappa^*(m_1) \kappa^*(m_2)
|J(m_1,m_2)|
\ll r^{1+\epsilon}.
\end{equation}
Combining the above two estimates together with \eqref{eq:14.13} we obtain
\begin{equation}
\sum_{\substack{0\le s_1<r\\ (s_1,r)=1}}\
\sum_{\substack{0\le s_2<r\\ (s_2,r)=1}}
V(k;s_1,s_2)
\ll (k,q)^{1/2}
q^{3/2+\epsilon}
r^{1+\epsilon}.
\end{equation}
This leads to \eqref{eq:14.9}, since
\begin{equation}
q^{1/2}
= (d_1/r)^{1/2}
< X^{1/4-21\varpi}
= X^{3/16+52\varpi}
\end{equation}
by the first inequality in \eqref{eq:13.12}
and
\begin{equation}
\sum_{k_1\sim K}\
\sum_{\substack{k_2\sim K\\ k_2\neq k_1}}
(k_1-k_1,q)^{1/2}
\ll q^\epsilon K^2,
\end{equation}
whence \eqref{eq:14.3} follows.

We now prove \eqref{eq:14.14}. We rewrite the left side of \eqref{eq:14.14} as
\begin{equation}
\frac{1}{r} 
\sum_{m_1=-\infty}^\infty\ \sum_{m_2=-\infty}^\infty\
\sum_{0\le k< r}
\kappa^*(m_1) \kappa^*(m_2+k)
|J(m_1,m_2+k)|.
\end{equation}
By \eqref{eq:14.10}, we have
\begin{equation}
\sum_{m=-\infty}^\infty
\kappa^*(m) \ll \mathcal{L},
\end{equation}
and $\kappa^*(m+k) \ll \kappa^* (m)$ for $0\le k<r$, since $r<q$ by \eqref{eq:13.11} and the second inequality in \eqref{eq:13.12}. Thus, to prove \eqref{eq:14.14}, it suffices to show that
\begin{equation} \label{eq:14.15}
\sum_{0\le k<r}
|J(m_1,m_2+k)|
\ll r^{2+\epsilon}
\end{equation}
for any $m_1$ and $m_2$. Substituting $s_2-s_1=t$ and applying M\"obius inversion we obtain
\begin{align} \label{eq:14.16}
J(m_1,m_2)
&= \sum_{|t|<r} C_r(t)
\sum_{\substack{s\in I_t\\ (s(s+t),r)=1}}
e_q\left(-\overline{r}(m_2t+(m_1+m_2)s) \right)\\
&\ll \sum_{|t|<r} |C_r(t)|
\left|\sum_{\substack{s\in I_t\\ s(s+t)\equiv 0 (r_1)}}
e_q(\overline{r}(m_1+m_2)s) \right|,
\end{align}
where $I_t$ is some interval of length less than $r$ depending on $t$. For any fixed $t$ and any square-free $r_1$, there are exactly $\tau(r_1/(t,r_1))$ distinct residue classes modulo $r_1$ such that
\begin{equation}
s(s+t) \equiv 0\quad (\textrm{mod } r_1)
\end{equation}
if and only if $s$ belongs to one of these classes. On the other hand, if $r=r_1r_2$, then
\begin{equation}
\sum_{\substack{s\in I_t\\ s\equiv a(r_1)}}
e_q\left(\overline{r} (m_1+m_2)s \right)
\ll \min\{r_2, \|\overline{r}_2(m_1+m_2)/q \|^{-1} \}
\end{equation}
for any $a$. Hence the inner sum on the right side of \eqref{eq:14.16} is
\begin{equation}
\ll \tau(r)
\sum_{r_2|r}
\min\{r_2, \|\overline{r}_2(m_1+m_2)/q \|^{-1} \},
\end{equation}
which does not depend on $t$. This, together with the simple estimate
\begin{equation}
\sum_{|t|<r} |C_r(t)| \ll \tau(r) r,
\end{equation}
yields
\begin{equation}
J(m_1,m_2) \ll \tau(r)^2 r
\sum_{r_2|r} 
\min\{r_2, \|\overline{r}_2(m_1+m_2)/q \|^{-1} \}.
\end{equation}
Thus the left side of \eqref{eq:14.15} is
\begin{equation} \label{eq:14.17}
\ll \tau(r)^2 r
\sum_{r_1r_2=r}\
\sum_{0\le k_1<r_1}\
\sum_{0\le k_2<r_2}
\min\{r_2, \|\overline{r}_2(m_1 + m_2+ k_1 r_2 + k_2)/q \|^{-1} \}.
\end{equation}
Assume $r_2|r$. By the relation
\begin{equation}
\frac{\overline{r_2}}{q}
\equiv -\frac{\overline{q}}{r_2} + \frac{1}{qr_2}\quad (\textrm{mod } 1),
\end{equation}
we have
\begin{equation} \label{eq:14.18}
\sum_{0\le k< r_2}
\min\{r_2, \|\overline{r_2}(m+k)/q \|^{-1} \}
\ll r_2 \mathcal{L}
\end{equation}
for any $m$. Hence the estimate \eqref{eq:14.15} follows at once from \eqref{eq:14.17} and \eqref{eq:14.18}.

\subsection{Proof of Case (c)}

We finish the proof of Theorem \ref{thm:maintheoremtauk} with the last and most involved case. This case corresponds to the Type I and II estimates in \cite[\S \S 7-12]{Zhang2014}. Without loss of generality, assume there is a subset $I$ of $\{1,2,\dots, k \}$ such that
\begin{equation} \label{eq:18}
	\frac{3}{8} + 8\varpi 
	< \sum_{i\in I} \nu_i
	< \frac{1}{2} +\frac{\log 2}{2\mathcal{L}}.
\end{equation}
Let $J$ be the complement of $I$ in $\{1,2,\dots, k \}$. Write
\begin{equation}
	\alpha = \beta_{j_1}* \beta_{j_2}* \dots * \beta_{j_m}, \quad
	J= \{j_1,j_2,\dots, j_m \},
\end{equation}
and
\begin{equation} \label{eq:beta}
	\beta = \beta_{i_1}*\beta_{i_2}*\dots *\beta_{i_\ell}, \quad
	I = \{i_1,i_2,\dots, i_\ell \},
\end{equation}
so that $\gamma=\alpha*\beta$. We have $\alpha$ supported on $[M,2M)$ and $\beta$ supported on $[N,2N)$, where
\begin{equation} \label{eq:betaN}
	M=\prod_{j\in J}N_j,\quad
	N=\prod_{i\in I} N_i.
\end{equation}
By \eqref{eq:18}, we have
\begin{equation} \label{eq:betaNbound}
	X^{3/8+8\varpi} < N \ll X^{1/2}.
\end{equation}

We now treat \eqref{eq:typeIandIISum} via the methods in \cite{FriedlanderIwaniec1985} and \cite[\S\S3-7]{BFI1986}, following \cite{Zhang2014}. Write
\begin{equation} \label{eq:2.13}
	X_1 = X^{3/8+8\varpi}
	\quad\text{and}\quad
	X_2 = X^{1/2-4\varpi}.
\end{equation}
We apply Lemma \ref{lemma:4} with
\begin{equation} \label{eq:7.1}
	R^* = 
	\begin{cases}
	X^{-\varpi/6} N, & \text{ if } X_1 < N \le X_2,\\
	X^{-3\varpi} N,& \text{ if } X_2 < N \le 2X^{1/2}.
	\end{cases}	
\end{equation}
Hence, by Lemma \ref{lemma:4}, the proof of Case (c) is reduced to showing that
\begin{equation} \label{eq:typeIandIISum}
	\sum_{\substack{q\sim Q}}
	\sum_{\substack{r \sim R\\ (r,a)=1\\ (q,r\mathcal{P}_0)=1}}
	\mu(qr)^2
	|\Delta(\gamma;X, qr,a)|
	\ll
	X^{1-\varpi^{5/3}}
\end{equation}
subject to the conditions
\begin{equation}
	X^{-\varpi}R^* < R < R^* 
	\quad \text{and} \quad
	\frac{1}{2} D_2 < QR < X^{1/2+2\varpi}.
\end{equation}

Therefore, it suffices to prove that
\begin{equation} \label{eq:BSum}
	\mathcal{B}(\gamma;Q,R)
	:= \sum_{\substack{r\sim R\\ (r,a)=1}} |\mu(r)|
	\sum_{\substack{q\sim Q\\ q|\mathcal{P}\\ (q,r\mathcal{P}_0)=1}}
	|\Delta(\gamma;X, qr,a)|
	\ll X^{1- \varpi^{5/3}}
\end{equation}
subject to the constraints
\begin{equation} \label{eq:7.4}
	X^{-\varpi} R^* < R < R^*
\end{equation}
and
\begin{equation} \label{eq:7.5}
	D_2 \ll QR \ll X^{1/2+2\varpi},
\end{equation}
which are henceforth assumed.

In what follows we assume that
\begin{equation} \label{eq:rConstraints}
	r\sim R,\quad
	|\mu(r)|=1,
	\quad\text{and}\quad
	(r,a)=1.
\end{equation}
Let $c(q,r)$ denote
\begin{equation}
	c(q,r) = 
	\begin{cases}
	\mathrm{sign}\Delta(\gamma; X, qr, a), & \text{if } q\sim Q, q| \mathcal{P}, \text{and } (q, r \mathcal{P}_0)=1,\\
	0,& \text{otherwise}.
	\end{cases}
\end{equation}
Splitting $\gamma=\alpha*\beta$, writing $n$ as $mn$, and changing the order of summation, the inner sum over $q$ in \eqref{eq:BSum} becomes
\begin{equation} \label{eq:qSum}
	\sum_{\substack{q\sim Q\\ q|\mathcal{P}\\ (q,r\mathcal{P}_0)=1}}
	|\Delta(\gamma;X, q,a)|
	= \sum_{(m,r)=1} \alpha(m) \mathcal{D}(r,m),
\end{equation}
where
\begin{equation}
	\mathcal{D}(r,m)
	:= \sum_{(q,m)=1} c(q,r)
	\left(\sum_{mn\equiv a (qr)} \beta(n)
	-  \frac{1}{\varphi(qr)} \sum_{(n,qr)=1}\beta(n) \right).
\end{equation}
Substituting \eqref{eq:qSum} into \eqref{eq:BSum} and applying Cauchy's inequality to the $m$ variable, we get
\begin{equation} \label{eq:7.7}
	\mathcal{B}(\gamma;Q,R)^2
	\ll MR\mathcal{L}^B
	\sum_{r\sim R} |\mu(r)|
	\sum_{(m,r)=1} f(m) \mathcal{D}(r,m)^2,
\end{equation}
where $f(y)$ is as in Lemma \ref{lemma:7}. Squaring out $\mathcal{D}(r,m)$ and summing over $m$, we have
\begin{equation} \label{eq:7.8}
	\sum_{(m,r)=1} f(m) \mathcal{D}(r,m)^2
	= \mathcal{S}_1(r)
	-2 \mathcal{S}_2(r)
	+ \mathcal{S}_3(r),
\end{equation}
where $\mathcal{S}_j(r,a)$, $j=1,2,3$, are defined by
\begin{align}
	\mathcal{S}_1(r) 
	&= \sum_{(m,r)=1} f(m)
	\left(\sum_{(q,m)=1} c(q,r) \sum_{mn\equiv a (qr)} \beta(n) \right)^2,\\
	\mathcal{S}_2(r)
	&= \sum_{q_1} \sum_{q_2}
	\frac{c(q_1,r) c(q_2,r)}{\varphi(q_2r)}
	\sum_{n_1} \sum_{(n_2,q_2r)=1} \beta(n_1)\beta(n_2)
	\sum_{\substack{mn_1\equiv a(q_1r)\\ (m,q_2)=1}} f(m),\\
	\mathcal{S}_3(r)
	&= \sum_{q_1} \sum_{q_2}
	\frac{c(q_1,r) c(q_2,r)}{\varphi(q_1r)\varphi(q_2r)}
	\sum_{(n_1,q_1r)=1} \sum_{(n_2,q_2r)=1} 
	\beta(n_1) \beta(n_2)
	\sum_{(m,q_1q_2r)=1} f(m).
\end{align}
By \eqref{eq:7.7} and \eqref{eq:7.8}, the proof of \eqref{eq:BSum} is reduced to showing that
\begin{equation} \label{eq:7.9}
	\sum_r
	(\mathcal{S}_1(r)
	-2 \mathcal{S}_2(r)
	+ \mathcal{S}_3(r))
	\ll NR^{-1}X^{1-\varpi^{5/3}},
\end{equation}
where $r$ is constrained as in \eqref{eq:rConstraints}.

We begin with the evaluation of $\mathcal{S}_3(r)$ which is the simplest of the three sums. We make frequent use of the trivial bound
\begin{equation} \label{eq:8.1}
	\hat{f}(z) \ll M.
\end{equation}
Similar to the proof of \eqref{eq:sharpsiegelwalfiszbound5}, we have, for $q_j\sim Q$, $j=1,2$,
\begin{equation}
	\sum_{(m,q_1q_2r)=1} f(m)
	= \frac{\varphi(q_1q_2r)}{q_1q_2r} \hat{f}(0)
	+ O(X^\epsilon).
\end{equation}
This yields
\begin{equation}
	\mathcal{S}_3(r)
	=\hat{f}(0)
	\sum_{q_1}\sum_{q_2} 
	\frac{c(q_1,r) c(q_2,r)}{\varphi(q_1r) \varphi(q_2r)}
	\frac{\varphi(q_1q_2r)}{q_1q_2r}
	\sum_{{(n_1,q_1r)=1}}
	\sum_{{(n_2,q_2r)=1}}
	\beta(n_1) \beta(n_2)
	+ O(X^\epsilon N^2R^{-2}).
\end{equation}
If $(q_1q_2,\mathcal{P}_0)=1$, then either $(q_1,q_2)=1$ or $(q_1,q_2)>D_0$. Thus, on the right side of the above, the contribution from terms with $(q_1,q_2)>1$ is, by \eqref{eq:8.1} and trivial estimation,
\begin{equation}
	\ll X N D_0^{-1} R^{-2} \mathcal{L}^B.
\end{equation}
It follows that
\begin{equation} \label{eq:8.3}
	\mathcal{S}_3(r)
	= \hat{f}(0) X(r) 
	+ O(XN D_0^{-1} R^{-2} \mathcal{L}^B),
\end{equation}
where
\begin{equation} \label{eq:X}
	X(r)
	= \sum_{q_1} \sum_{(q_2,q_1)=1}
	\frac{c(q_1,r) c(q_2,r)}{q_1q_2r \varphi(r)}
	\sum_{{(n_1,q_1r)=1}}
	\sum_{{(n_2,q_2r)=1}}
	\beta(n_1) \beta(n_2).
\end{equation}
The value of $X(r)$ is not essential, since it will cancel out, with acceptable error (see \eqref{eq:10.15} below), when we insert it back into the left side of \eqref{eq:7.9}. We next evaluate $\mathcal{S}_2(r)$.

Our next goal is to show that
\begin{equation} \label{eq:9.1}
	\mathcal{S}_2(r)
	= \hat{f}(0) X(r)
	+ O(XN D_0^{-1} R^{-2} \mathcal{L}^B)
\end{equation}
with $X(r)$ given as in \eqref{eq:X}. The main tool we need is the Ramanujan bound \eqref{eq:3.13}. Assume $c(q_1,r) c(q_2,r)\neq 0$. On substituting $mn_1=n$ and applying Lemma \ref{lemma:8} we have
\begin{equation}
	\sum_{n_1} \beta(n_1)
	\sum_{\substack{mn_1\equiv a (q_1r)\\ (m,q_2)=1}} f(m)
	\ll \sum_{\substack{n< 2X\\ n\equiv a(q_1r)}} \tau_k(n)
	\ll \frac{X\mathcal{L}^B}{q_1r}.
\end{equation}
It follows that contributions from terms with $(q_1,q_2)>1$ in $\mathcal{S}_2(r)$ is
\begin{equation}
	\ll XN D_0^{-1} R^{-2} \mathcal{L}^B,
\end{equation}
so that
\begin{equation} \label{eq:9.2}
	\mathcal{S}_2(r)
	= \sum_{q_1} \sum_{(q_2,q_1)=1}
	\frac{c(q_1,r) c(q_2,r)}{\varphi(q_2r)}
	\sum_{n_1} \sum_{(n_2,q_2r)=1} \beta(n_1) \beta(n_2)
	\sum_{\substack{mn_1\equiv a(q_1r)\\ (m,q_2)=1}} f(m)
	+O(XN D_0^{-1} R^{-2} \mathcal{L}^B).
\end{equation}
Note that the innermost sum over $m$ in \eqref{eq:9.2} is empty unless $(n_1,q_1r)=1$. For $|\mu(q_1q_2r)|=1$ and $(q_2,\mathcal{P}_0)=1$ we have
\begin{equation}
	\frac{q_2}{\varphi(q_2)}
	= 1+ O(\tau(q_2) D_0^{-1})
\end{equation}
and, by Lemma \ref{lemma:8},
\begin{equation}
	\sum_{(n_1,q_1r)=1} \beta(n_1)
	\sum_{\substack{mn_1\equiv a (q_1r)\\ (m,q_2)>1}} f(m)
	\ll \sum_{\substack{n< 2X\\ n\equiv a(q_1r)\\ (n,q_2)>1}} \tau_k(n)
	\ll \frac{\tau_k(q_2)X\mathcal{L}^B}{q_1r D_0}.
\end{equation}
Thus \eqref{eq:9.2} still holds with the constraint $(m,q_2)=1$ removed and with $\varphi(q_2r)$ replaced by $q_2\varphi(r)$. That is, we have
\begin{equation} \label{eq:9.3}
	\mathcal{S}_2(r)
	= \sum_{q_1} \sum_{(q_2,q_1)=1}
	\frac{c(q_1,r) c(q_2,r)}{q_2\varphi(r)}
	\sum_{n_1} \sum_{(n_2,q_2r)=1} \beta(n_1)\beta(n_2)
	\sum_{\substack{mn_1\equiv a(q_1r)}} f(m)
	+O(XN D_0^{-1} R^{-2} \mathcal{L}^B).
\end{equation}
By Lemma \ref{lemma:7}, for $(n_1,q_1r)=1$, we have
\begin{equation}
	\sum_{mn_1\equiv a (q_1r)} f(m)
	=\frac{1}{q_1 r}
	\sum_{|h|< H_2} \hat{f}\left(\frac{h}{q_1r} \right)
	e_{q_1r}(-hm) + O(d^{-1}),
\end{equation}
where
\begin{equation}
	H_2 = 4QRM^{-1+2\epsilon}.
\end{equation}
Substituting this into \eqref{eq:9.3} we deduce that
\begin{equation} \label{eq:9.5}
	\mathcal{S}_2(r)
	= \hat{f}(0) X(r)
	+ \mathcal{R}_2(r)
	+ O(XN D_0^{-1} R^{-2} \mathcal{L}^B),
\end{equation}
where
\begin{align}
	\mathcal{R}_2(r)
	= &\sum_{q_1} \sum_{(q_2,q_1)=1}
	\frac{c(q_1,r) c(q_2,r)}{q_1q_2\varphi(r)}
	\left(\sum_{(n_2,q_2r)=1}\beta(n_2)\right)\\
	 &\times \sum_{(n_1,q_1r)=1}\beta(n_1)
	 \sum_{1\le |h|< H_2} \hat{f}\left(\frac{h}{q_1r}\right)
	 e_{q_1r}(-hm).
\end{align}
The proof of \eqref{eq:9.1} is now reduced to estimating $\mathcal{R}_2(r)$. Note that, by the second inequality in \eqref{eq:7.5}, we have
\begin{equation} \label{eq:9.6}
	H_2 \ll X^{-1/2+2\varpi+2\epsilon} N
\end{equation}
since $M^{-1}\ll X^{-1}N$. This implies that $\mathcal{R}_2(r)=0$ if $X_1< N< X_2$, since $H_2<1$ in this case.

Now assume that $X_2 < N < 2X^{1/2}$. 
By the `reciprocity' relation
\begin{equation}
	\frac{m}{q_1r}
	\equiv \frac{a\overline{q_1n_1}}{r}
	+ \frac{a\overline{rn_1}}{q_1}
	\quad (\text{mod } 1),
\end{equation}
we get
\begin{equation} \label{eq:9.7}
	\mathcal{R}_2(r)
	\ll N^{1+\epsilon} R^{-2}
	\sum_{\substack{n\sim N\\ (n,r)=1}}
	|\mathcal{R}^*(r,n)|,
\end{equation}
where
\begin{equation}
	\mathcal{R}^*(r,n)
	= \sum_{(q,n)=1} \frac{c(q,r)}{q}
	\sum_{1\le |h|< H_2}
	\hat{f}\left(\frac{h}{qr}\right)
	e\left(
	\frac{-ah\overline{qn}}{r}
	- \frac{ah\overline{rn}}{q}
	\right).
\end{equation}
Expanding $|\mathcal{R}^*(r,n)|^2$ out, we have
\begin{align}
	|\mathcal{R}^*(r,n)|^2
	&= \sum_{(q,n)=1} \sum_{(q',n)=1}
	\frac{c(q,r) c(q',r) }{qq'}\\
	&\times \sum_{1\le|h|< H_2}\ \sum_{1\le|h'|< H_2}
	\hat{f}\left(\frac{h}{qr} \right)
	\overline{\hat{f}\left(\frac{h'}{q'r} \right)}
	e\left(
	\frac{a(h'\overline{q}'-h\overline{q} \overline{n}}{r} - \frac{ah\overline{rn}}{q} + \frac{ah' \overline{rn}}{q'}
	\right).
\end{align}
Changing the order of summation and applying \eqref{eq:8.1}, we obtain
\begin{align} \label{eq:9.8}
	M^{-2}
	\sum_{\substack{n\sim N\\ (n,r)=1}}
	|\mathcal{R}^*(r,n)|^2
	&\ll
	\sum_q \sum_{q'}
	\frac{|c(q,r) c(q',r) |}{qq'}
	\sum_{1\le |h|< H_2}\ \sum_{1\le |h'|<H_2}
	|\mathcal{W}(q,r;q',h')|,
\end{align}
where
\begin{equation}
	\mathcal{W}(q,r;q',h')
	= \sum_{\substack{n\sim N\\ (n,qq'r)=1}}
	e\left(
	\frac{a(h'\overline{q}'-h\overline{qn})}{r} - \frac{ah\overline{rn}}{q} + \frac{ah' \overline{rn}}{q'}
	\right).
\end{equation}

Since $M^{-1}\ll N^{-1}$, by the second inequality in \eqref{eq:7.4} we have
\begin{equation} \label{eq:9.9}
	H_2Q^{-1} \ll X^{-3\varpi+\epsilon}.
\end{equation}
It follows that, on the right side of \eqref{eq:9.8}, the contribution from terms with $h'q=hq'$ is
\begin{equation} \label{eq:9.10}
	\ll NQ^{-2} 
	\sum_{1\le |h|< H_2} \sum_{q<2Q} \tau(hq)
	\ll X^{-3\varpi + \epsilon}N.
\end{equation}

Now assume that $c(q,r) c(q',r) \neq 0$, $1\le |h|< H_2$, $1\le |h'|< H_2$, and $h'q\neq hq'$. Letting $d=[q,q']r$, we have
\begin{equation}
	\frac{a(h'\overline{q}'-h\overline{q})}{r} - \frac{ah\overline{r}}{q} + \frac{ah' \overline{r}}{q'}
	\equiv \frac{c}{d}\quad
	(\text{mod }1)
\end{equation}
for some $c$ with
\begin{equation}
	(c,r) = (h'\overline{q}' - h\overline{q},r).
\end{equation}
By the estimate \eqref{eq:3.13} it follows that
\begin{equation} \label{eq:9.11}
	\mathcal{W}(q,r;q',h')
	\ll d^{1/2+\epsilon} + \frac{(c,d)N}{d}.
\end{equation}
Since $N>X_2$, by the first inequality in \eqref{eq:7.4}, \eqref{eq:7.1}, and $\eqref{eq:2.13}$, we have
\begin{equation} \label{eq:9.12}
	R^{-1} < X^{4\varpi}N^{-1} < X^{-1/2+8\varpi}.
\end{equation}
This and the second inequality in \eqref{eq:7.5} imply that
\begin{equation} \label{eq:9.13}
	Q \ll X^{10\varpi}.
\end{equation}
Thus we have
\begin{equation}
	d^{1/2} \ll (Q^2R)^{1/2} \ll X^{1/4+6\varpi}.
\end{equation}
Note that 
\begin{equation}
	h'\overline{q}' - h\overline{q} 
	\equiv (h'q - hq') \overline{qq'}
	\quad (\text{mod } r).
\end{equation}
This implies
\begin{equation} \label{eq:9.14}
	(c,d) \le (c,r) [q,q'] \ll [q,q'] H_2Q.
\end{equation}
This together with \eqref{eq:9.6}, \eqref{eq:9.12}, and \eqref{eq:9.13}, give
\begin{equation}
	\frac{(c,d)N}{d}
	\ll H_2NQR^{-1}
	\ll X^{16\varpi+\epsilon}.
\end{equation}
Combining these estimates with \eqref{eq:9.11} we deduce that
\begin{equation}
	\mathcal{W}(q,r;q',h')
	\ll X^{1/4+7\varpi}.
\end{equation}
This together with \eqref{eq:9.6} imply that the contribution from terms with $h'q\neq hq'$ on the right side of \eqref{eq:9.8} is $\ll X^{1/4+12\varpi}$, which is sharper than the right side of \eqref{eq:9.10}. Combining these estimates with \eqref{eq:9.8} we conclude that
\begin{equation}
	\sum_{\substack{n\sim N\\ (n,r)=1}}
	|\mathcal{R}^*(r,n)|
	\ll X^{1-3\varpi/2+\epsilon}.
\end{equation}
Substituting this into \eqref{eq:9.7} we obtain
\begin{equation} \label{eq:9.15}
	\mathcal{R}_2(r) \ll NR^{-2} X^{1-\varpi},
\end{equation}
which is sharper than the big-Oh term in \eqref{eq:9.5}.

The relation \eqref{eq:9.1} follows from the bounds \eqref{eq:9.5} and \eqref{eq:9.15} immediately. It remains to deal with $\mathcal{S}_1(r)$. The evaluation of the last sum $\mathcal{S}_1(r)$ is the most difficult. The main tool we need is Lemma \ref{lemma:11}.

We shall instead establish an averaged bound on $\mathcal{S}_1(r)$ of the form
\begin{equation} \label{eq:10.1}
	\sum_r \mathcal{S}_1(r)
	= \sum_r (\hat{f}(0)X(r) + \mathcal{R}_1(r))
	+ O(XN R^{-1} X^{-\varpi^{5/3}})
\end{equation}
with $\mathcal{R}_1(r)$ to be specified below in \eqref{eq:10.10}. By the estimates \eqref{eq:9.1} for $\mathcal{S}_2(r)$ and \eqref{eq:8.3} for $\mathcal{S}_3(r)$, the proof of \eqref{eq:7.9} will be reduced to estimating $\mathcal{R}_1(r)$.

By definition of $\mathcal{S}_1(r)$, expanding out the square, we have
\begin{equation} \label{eq:10.2}
	\mathcal{S}_1(r)
	= \sum_{q_1} \sum_{q_2} c(q_1,r) c(q_2,r)
	\sum_{n_1} \sum_{n_2\equiv n_1 (r)}
	\beta(n_1) \beta(n_2)
	\sum_{\substack{mn_1\equiv a (q_1r)\\ mn_2\equiv a (q_2)}} f(m).
\end{equation}
Let $\mathcal{U}(r,q_0)$ denote the sum of terms on the right side of the above with $(q_1,q_2)=q_0$. The sum $\mathcal{U}(r,q_0)$ vanishes unless
\begin{equation}
	q_0< 2Q,\quad
	q_0| \mathcal{P},\quad
	\text{and}\quad
	(q_0,r\mathcal{P}_0) =1,
\end{equation}
which we will assume. We first show that the contribution 
\begin{equation} \label{eq:10.3}
	\sum_r \sum_{q_0>1}
	\mathcal{U}(r,q_0)
	\ll XN (D_0 R)^{-1} \mathcal{L}^B
\end{equation}
coming from terms with $q_0>1$ is admissible. Assume that, for $j=1,2$,
\begin{equation}
	q_j\sim Q,\quad
	q_j|\mathcal{P},\quad
	(q_j,r\mathcal{P}_0)=1,\quad
	\text{and}\quad (q_1,q_2) = q_0.
\end{equation}
Writing 
$q_1'=q_1/q_0$, $q_2'=q_2/q_0$, we get
\begin{equation}
	\sum_r \sum_{q_0>1}
	\mathcal{U}(r,q_0) = 
	\sum_r \sum_{q_0>1}
	\mathop{\sum\sum}_{\substack{(q',q'')=1\\ (a, q'q'')=1}}
	c(q_0q',r) c(q_0q'',r)
	\sum_{n_1}
	\sum_{n_2\equiv n_1 (r)}
	\beta(n_1) \beta(n_2)
	\sum_{m\equiv \mu (q_1q_2r)}
	f(m)
\end{equation}
where $\mu (\text{mod } q_1q_2r)$ is a common solution to
\begin{equation} \label{eq:10.8}
	\mu n_1 \equiv a (\text{mod } q_1r)\quad
	\text{and}\quad
	\mu n_2 \equiv a (\text{mod } q_2r).
\end{equation}
Since $q_1$ and $q_2$ have no prime factor less than $D_0$, we have either $q_0=1$ or $q_0\ge D_0$. By Cauchy's inequality we have
\begin{equation}
	\sum_r \sum_{q_0>1}
	\mathcal{U}(r,q_0)
	\ll
	\sum_{D_0 < q_0 \le 2Q}
	\sum_r
	\sum_m f(m)
	\sum_{q'} 
	\sum_{n_1 \equiv a \overline{m} (q'r)}
	|\beta(n_1)|^2
	\sum_{q_2}
	\sum_{n_2\equiv a \overline{m} (q_2r)}
	1.
\end{equation}
By Lemma \ref{lemma:8} the two inner sums of the above is
\begin{equation}
	\ll \sum_{n \equiv a \overline{m} (q_0r)}
	\tau^B(mn-a) \tau^B(q_0)
	\ll N(D_0R)^{-1} \mathcal{L}^B.
\end{equation}
This and Lemma \ref{lemma:8} imply
\begin{equation}
	\sum_r \sum_{q_0>1}
	\mathcal{U}(r,q_0)
	\ll N(D_0 R)^{-1} \mathcal{L}^B
	\sum_m \sum_n
	\tau^B(mn-a)
	\ll N R^{-1} \mathcal{L}^BX D_0^{-1}.
\end{equation}
This bound is admissible provided 
\begin{equation} \label{eq:D02}
	D_0 = X^{\varpi^{4/3}}
\end{equation}
which we henceforth assume.

We now turn to $\mathcal{U}(r,1)$. Assume $|\mu(q_1q_2r)|=1$. In the case $(n_1,q_1r)=(n_2,q_2r)=1$, the innermost sum in \eqref{eq:10.2} is, by Lemma \ref{lemma:7}, equal to
\begin{equation}
	\frac{1}{q_1q_2r}
	\sum_{|h|< H_1}
	\hat{f} \left(\frac{h}{q_1q_2r} \right)
	e_{q_1q_2r} (-\mu h) + O(d^{-1}),
\end{equation}
where
\begin{equation} \label{eq:10.7}
	H_1 = 8Q^2RM^{-1+2\epsilon}.
\end{equation}
It follows that
\begin{equation} \label{eq:10.9}
	\mathcal{U}(r,1)
	= \hat{f}(0) X^*(r) + \mathcal{R}_1(r) +O(1),
\end{equation}
where
\begin{equation}
	X^*(r)
	= \sum_{q_1} \sum_{(q_2,q_1)=1}
	\frac{c(q_1,r) c(q_2,r)}{q_1q_2r}
	\sum_{{(n_1,q_1r)=1}}
	\sum_{\substack{n_2\equiv n_1 (r)\\ (n_2,q_2)=1}}
	\beta(n_1) \beta(n_2)
\end{equation}
and
\begin{equation} \label{eq:10.10}
	\mathcal{R}_1(r)
	= \sum_{q_1} \sum_{(q_2,q_1)=1}
	\frac{c(q_1,r) c(q_2,r)}{q_1q_2r}
	\sum_{\substack{n_2\equiv n_1 (r)\\ (n_2,q_2)=1}}
	\beta(n_1) \beta(n_2)
	\sum_{1\le |h|< H}
	\hat{f}\left(\frac{h}{q_1q_2r} \right)
	e_{q_1q_2r}(-\mu h).
\end{equation}
By \eqref{eq:10.2}, \eqref{eq:10.3}, and \eqref{eq:10.9} we conclude that
\begin{equation}
	\sum_r \mathcal{S}_1(r)
	=\sum_r (\hat{f}(0) X^*(r) + \mathcal{R}_1(r))
	+ O(XN (D_0 R)^{-1} \mathcal{L}^B).
\end{equation}
In view of \eqref{eq:8.1}, the proof of \eqref{eq:10.1} is thus reduced to showing that
\begin{equation} \label{eq:10.11}
	\sum_r (X^*(r) - X(r))
	\ll N^2R^{-1} X^{-\varpi^{5/3}}.
\end{equation}

We have
\begin{equation}
	X^*(r) - X(r)
	= \sum_{q_1} \sum_{(q_2,q_1)=1}
	\frac{c(q_1,r) c(q_2,r)}{q_1q_2r}
	\mathcal{V}(r;q_1,q_2),
\end{equation}
with
\begin{equation}
	\mathcal{V}(r;q_1,q_2)
	= \sum_{(n_1,q_1r)=1}
	\sum_{\substack{n_2\equiv n_1(r)\\ (n_2,q_2)=1}}
	\beta(n_1)\beta(n_2)
	- \frac{1}{\varphi(r)}
	\sum_{(n_1,q_1r)=1} \sum_{(n_2,q_2r)=1}
	\beta(n_1)\beta(n_2).
\end{equation}
It follows that
\begin{equation} \label{eq:10.12}
	\sum_r (X^*(r) - X(r))
	\ll \frac{1}{R}
	\sum_{q_1\sim Q} \sum_{q_2\sim Q}
	\frac{1}{q_1q_2}
	\sum_{\substack{r\sim R\\ (r,q_1q_2)=1}}
	|\mathcal{V}(r;q_1,q_2)|.
\end{equation}
Noting that
\begin{equation}
	\mathcal{V}(r;q_1,q_2)
	= \displaystyle\sideset{}{^*}\sum_{\ell (\text{mod } r)}
	\left(\sum_{\substack{n\equiv \ell (r)\\ (n,q_1)=1}} \beta(n)
	- \frac{1}{\varphi(r)} \sum_{(n,q_1r)=1}\beta(n) \right)
	\left(\sum_{\substack{n\equiv \ell (r)\\ (n,q_2)=1}} \beta(n)
	- \frac{1}{\varphi(r)} \sum_{(n,q_2r)=1}\beta(n) \right),
\end{equation}
and by Cauchy's inequality and Lemma \ref{lemma:10}, we find that the innermost sum in \eqref{eq:10.12} is
\begin{equation}
	\ll \tau(q_1q_2)^B N^2 X^{-\varpi/12},
\end{equation}
which leads to \eqref{eq:10.11}.

Combining \eqref{eq:8.3}, \eqref{eq:9.1}, and \eqref{eq:10.1} leads to
\begin{equation} \label{eq:10.13}
	\sum_r (\mathcal{S}_1(r) - 2\mathcal{S}_2(r) + \mathcal{S}_3(r))
	= \sum_r 
	\mathcal{R}_1(r)
	+ O(XN R^{-1} X^{-\varpi^{5/3}}).
\end{equation}
Note that
\begin{equation}
	\frac{\mu}{q_1q_2r}
	\equiv
	\frac{a \overline{q_1q_2 n_1}}{r}
	+ \frac{a \overline{q_2rn_1}}{q_1}
	+ \frac{a \overline{q_1rn_2}}{q_2}\quad
	(\text{mod } 1)
\end{equation}
by \eqref{eq:10.8}. Hence, on substituting $n_2=n_1+kr$, we may write $\mathcal{R}_1(r)$ as
\begin{equation} \label{eq:10.14}
	\mathcal{R}_1(r)
	= \frac{1}{r} \sum_{|k|<N/R}
	\mathcal{R}_1(r,k),
\end{equation}
where
\begin{align}
	\mathcal{R}_1(r,k)
	=&\sum_{q_1} \sum_{(q_2,q_1)=1}
	\frac{c(q_1,r) c(q_2,r)}{q_1q_2}
	\sum_{1\le |h|< H_1}
	\hat{f}\left(\frac{h}{q_1q_2r} \right)\\
	&\times
	\sum_{\substack{(n,q_1r)=1\\ (n+kr,q_2)=1}}
	\beta(n) \beta(n+kr)
	e(-h\xi(r; q_1, q_2; n, k)),
\end{align}
with
\begin{equation}
	\xi(r; q_1, q_2; n, k) = 
	\frac{a \overline{q_1q_2 n}}{r}
	+ \frac{a \overline{q_2rn}}{q_1}
	+ \frac{a \overline{qr(n+ kr)}}{q_2}.
\end{equation}
Thus, the proof of \eqref{eq:7.9} will follow from the bound
\begin{equation} \label{eq:10.15}
	\mathcal{R}_1(r,k)
	\ll X^{1-\varpi/2}
\end{equation}
which we shall prove in the next two subsections. We note that the bound \eqref{eq:10.15} amounts to saving a power of $X$ from the trivial estimate; indeed, it trivially follows from \eqref{eq:8.1} that
\begin{equation}
	\mathcal{R}_1(r,k)
	\ll X^{1+\epsilon} H_1.
\end{equation}
On the other hand, in view of \eqref{eq:2.13}, since
\begin{equation}
	H_1 \ll X^\epsilon (QR)^2 (MN)^{-1} NR^{-1},
\end{equation}
and, by the first inequality in \eqref{eq:7.4} and \eqref{eq:7.1},
\begin{equation} \label{eq:10.16}
	NR^{-1} <
	\begin{cases}
	X^{\varpi + \epsilon}, & \text{if } X_1<N\le X_2,\\
	X^{4\varpi}, & \text{if } X_2< N< 2X^{1/2},
	\end{cases}
\end{equation}
it follows from the second inequality in \eqref{eq:7.5} that
\begin{equation} \label{eq:10.17}
	H_1 \ll
	\begin{cases}
	X^{5\varpi + 2\epsilon} & \text{if } X_1<N\le X_2,\\
	X^{8\varpi + \epsilon} & \text{if } X_2<N\le 2X^{1/2}
	\end{cases}
\end{equation}
is bounded by a small power of $X$.

\subsubsection{Estimation of $\mathcal{R}_1(r,k)$: The Type I case}

In this and the next subsection we assume that $|k|<NR^{-1}$, and we write
\begin{equation}
	\mathcal{R}_1,\quad
	c(q_1),\quad
	c(q_2),\quad
	\text{and}\quad
	\xi(q_1,q_2;n)
\end{equation}
for
\begin{equation}
	\mathcal{R}_1(r,k),\quad
	c(q_1,r),\quad
	c(q_2,r),\quad
	\text{and}\quad
	\xi(r; q_1, q_2; n, k),
\end{equation}
respectively, with the goal of proving \eqref{eq:10.15}. The variables $r$ and $k$ may also be omitted for notational simplicity. The proof is  analogous to the estimation of $\mathcal{R}_2(r)$; the main ingredient is Lemma \ref{lemma:11}.

Assume that $X_1<N\le X_2$ and $R^*$ is as in \eqref{eq:7.1}. We have
\begin{equation} \label{eq:11.1}
	\mathcal{R}_1
	\ll N^\epsilon
	\sum_{q_1} \frac{c(q_1)}{q_1} 
	\sum_{\substack{n\sim N\\ (n,q_1r)=1}}
		|\mathcal{F}(q_1,n)|,
\end{equation} 
where
\begin{equation}
	\mathcal{F}(q_1,n)
	= \sum_{0< |h| < H_1}
	\sum_{(q_2,q_1(n+kr))=1}
	\frac{c(q_2)}{q_2}
	\hat{f} \left(\frac{h}{q_1q_2r} \right)
	e(-h\xi(q_1,q_2;n)).
\end{equation}
We assume $c(q_1)\neq 0$. To bound the sum of $|\mathcal{F}(q_1,n)|$ we observe that, similar to \eqref{eq:9.8},
\begin{equation} \label{eq:11.2}
	M^{-2}\sum_{\substack{n\sim N\\ (n,q_1r)=1}}
	|\mathcal{F}(q_1,n)|^2
	\ll 
	\sum_{(q_2,q_1)=1} \sum_{(q_2',q_1)=1}
	\frac{|c(q_2)c(q_2')|}{q_2q_2'}
	\sum_{0< |h|< H_1} \sum_{0< |h'|< H_1}
	|\mathcal{G}(h,h'; q_1, q_2; q_2')|,
\end{equation}
where
\begin{equation}
	\mathcal{G}(h,h'; q_1, q_2; q_2')
	= \sum_{\substack{n\sim N\\ (n,q_1r)=1\\ (n+kr,q_2q_2')=1}}
	e(h'\xi(q_1, q_2';n) 
	- h\xi(q_1, q_2;n)).
\end{equation}
The condition $N\le X_2$ is essential for bounding the diagonal terms $h'q_2=hq_2'$ in \eqref{eq:11.2}. By \eqref{eq:7.5} we have
\begin{equation}
	H_1Q^{-1}
	\ll X^\epsilon(QR) (MN)^{-1} N
	\ll X^{-2\varpi+\epsilon}.
\end{equation}
It follows that, on the right side of \eqref{eq:11.2}, the contribution from terms with $h'q_2=hq_2'$ is
\begin{equation} \label{eq:11.3}
	\ll NQ^{-2}
	\sum_{1\le |h|< H_1}
	\sum_{q\sim Q}
	\tau(hq)^B
	\ll X^{-2\varpi+\epsilon} N.
\end{equation}

Now assume that $c(q_2)c(q_2')\neq 0$, $(q_2q_2',q_1)=1$, and $h'q_2\neq hq_2'$. We have
\begin{align}
	h'\xi(q_1, q_2';n) 
	- h\xi(q_1, q_2;n)
	\equiv \
	&\frac{h' \overline{q_2'} - h \overline{q_2} a \overline{q_1 n}}{r}\\
	&+ \frac{h'\overline{q_2'} - h \overline{q_2} a \overline{rn}}{q_1}
	+ \frac{h' a \overline{q_1 r(n+kr)}}{q_2'}
	- \frac{ha \overline{q_1 r(n_kr)}}{q_2}\
	(\text{mod } 1).
\end{align}
Letting $d_1=q_1r$ and $d_2=[q_2,q_2']$, we may write
\begin{equation}
	\frac{h'\overline{q_2'} - h\overline{q_2} a \overline{q_1}}{r}
	+ \frac{h'\overline{q_2'} - h\overline{q_2} a \overline{r}}{q_1}
	\equiv \frac{c_1}{d_1}\ (\text{mod }1)
\end{equation}
for some $c_1$ with
\begin{equation}
	(c_1,r) = (h'\overline{q_2'} - h \overline{q_2},r),
\end{equation}
and
\begin{equation}
	\frac{h' a \overline{q_1 r}}{q_2'}
	- \frac{ha \overline{q_1r}}{q_2}
	\equiv \frac{c_2}{d_2}\ (\text{mod }1)
\end{equation}
for some $c_2$, so that
\begin{equation}
	h'\xi(q_1, q_2';n) 
	- h\xi(q_1, q_2;n)
	\equiv \frac{c_1\overline{n}}{d_1} + \frac{c_2 \overline{(n+kr)}}{d_2}\
	(\text{mod }1).
\end{equation}
Since $(d_1,d_2)=1$, it follows by Lemma \ref{lemma:11} that
\begin{equation} \label{eq:11.4}
	\mathcal{G}(h,h'; q_1, q_2; q_2')
	\ll (d_1d_2)^{1/2+\epsilon}
	+ \frac{(c_1,d_1)N}{d_1}.
\end{equation}
By the condition $N>X_1$, this gives, by \eqref{eq:10.16},
\begin{equation} \label{eq:11.5}
	R^{-1} < X^{\varpi+\epsilon} N^{-1}
	< X^{-3/4-15\varpi+\epsilon} N.
\end{equation}
Together with \eqref{eq:7.5}, this yields
\begin{equation}
	(d_1d_2)^{1/2}
	\ll (Q^3R)^{1/2}
	\ll X^{3/4+3\varpi} R^{-1}
	\ll X^{-12\varpi+\epsilon}N.
\end{equation}
A sharper bound for the second term on the right side of \eqref{eq:11.4} can be obtained as follows. In a way similar to the proof of \eqref{eq:9.14}, we find that
\begin{equation}
	(c_1,d_1) \le (c_1,r) q_1 \ll H_1 Q^2.
\end{equation}
It follows by \eqref{eq:10.17}, \eqref{eq:7.5}, and the first inequality in \eqref{eq:11.5} that
\begin{equation}
	\frac{(c_1,d_1)}{d_1}
	\ll H_1 (QR) R^{-2}
	\ll X^{1/2+9\varpi+4\epsilon} N^{-2}
	\ll X^{-1/4-6\varpi}.
\end{equation}
Here the condition $N>X_1$ is used again. Combining these estimates with \eqref{eq:11.4} we deduce that
\begin{equation}
	\mathcal{G}(h,h'; q_1, q_2; q_2')
	\ll X^{-12\varpi+\epsilon} N.
\end{equation}
Together with \eqref{eq:10.17}, this implies that, on the right side of \eqref{eq:11.2}, the contribution from terms with $h'q_2\neq hq_2'$ is
\begin{equation}
	\ll X^{-12 \varpi+\epsilon} H_1^2 N
	\ll X^{-2\varpi + 5\epsilon} N,
\end{equation}
which has the same order of magnitude as the right side of \eqref{eq:11.3} essentially. Combining these estimates with \eqref{eq:11.2} we obtain
\begin{equation}
	\sum_{\substack{n\sim N\\ (n,q_1r)=1}}
	|\mathcal{F}(q_1,n)|^2
	\ll X^{1-2\varpi+5\epsilon}M.
\end{equation}
This yields, by Cauchy's inequality,
\begin{align} \label{eq:11.6}
	\sum_{\substack{n\sim N\\ (n,q_1r)=1}}
	|\mathcal{F}(q_1,n)|
	&\ll \left(\sum_{\substack{n\sim N\\ (n,q_1r)=1}} 1\right)^{1/2}
	\left(\sum_{\substack{n\sim N\\ (n,q_1r)=1}}
	|\mathcal{F}(q_1,n)|^2 \right)^{1/2}\\
	&\ll N^{1/2}(X^{1-2\varpi + 5\epsilon}M)^{1/2}
	\ll X^{1-\varpi+3\epsilon}.
\end{align}
The estimate \eqref{eq:10.15} follows from \eqref{eq:11.1} and \eqref{eq:11.6} immediately.

\subsubsection{Estimation of $\mathcal{R}_1(r,k)$: The Type II case}
We now assume that $X_2<N<2X^{1/2}$ and $R^*$ is as in \eqref{eq:7.1}. We have
\begin{equation} \label{eq:12.1}
	\mathcal{R}_1 \ll N^\epsilon
	\sum_{\substack{n\sim N\\ (n,r)=1}}
	|\mathcal{K}(n)|,
\end{equation}
where
\begin{equation}
	\mathcal{K}(n)
	= \sum_{(q_1,n)=1}
	\sum_{(q_2,q_1(n+kr))=1}
	\frac{c(q_1)c(q_2)}{q_1q_2}
	\sum_{1\le |h|< H_1}
	\hat{f} \left(\frac{h}{q_1q_2r} \right)
	e(-h \xi(q_1,q_2;n)).
\end{equation}
Let $\displaystyle\sideset{}{^\#}\sum$ stands for a summation over the tuples $(q_1,q_2; q_1', q_2')$ with
\begin{equation}
	(q_1,q_2) = (q_1',q_2') = 1.
\end{equation}
To estimate the sum of $\mathcal{K}(n)$ we observe that, similar to \eqref{eq:9.8},
\begin{equation} \label{eq:12.2}
	M^{-2}
	\sum_{\substack{n\sim N\\ (n,r)=1}}
	|\mathcal{K}(n)|^2
	\ll \displaystyle\sideset{}{^\#}\sum
	\frac{c(q_1)c(q_2)c(q_1')c(q_2')}{q_1q_2q_1'q_2'}
	\sum_{1\le |h|< H_1}
	\sum_{1\le |h'|< H_1}
	|\mathcal{M}(h,h'; q_1, q_2; q_1', q_2')|,
\end{equation}
where
\begin{equation}
	\mathcal{M}(h,h'; q_1, q_2; q_1', q_2')
	= \displaystyle\sideset{}{^\prime}\sum_{n\sim N}
	e(h' \xi(q_1,q_2;n) - h \xi(q_1,q_2;n)).
\end{equation}
Here $\displaystyle\sideset{}{^\prime}\sum$ denote a sum is restricted to $(n,q_1q_1'r) = (n+kr, q_2q_2')=1$.

Similar to \eqref{eq:9.9}, we have
\begin{equation}
	H_1Q^{-2} \ll X^{-3\varpi+\epsilon}.
\end{equation}
Hence, on the right side of \eqref{eq:12.2}, the contribution from terms with $h'q_1q_2=hq_1'q_2'$ is
\begin{equation} \label{eq:12.3}
	\ll NQ^{-4}
	\sum_{1\le |h|< H_1}
	\sum_{q\sim Q} \sum_{q'\sim Q}
	\tau(hqq')^B
	\ll X^{-3\varpi+\epsilon}N.
\end{equation}

Note that the bounds \eqref{eq:9.12} and \eqref{eq:9.13} are valid in this present situation. Since $R$ is slightly smaller than $X^{1/2}$ and $Q$ is small, by Lemma \ref{lemma:11}, contribution from terms with $h'q_1q_2 \neq hq_1'q_2'$ on the right side of \eqref{eq:12.2} is small compare to \eqref{eq:12.3}. Assume that
\begin{equation}
	c(q_1)c(q_2)c(q_1')c(q_2') \neq 0,\quad
	(q_1,q_2) = (q_1',q_2')=1,\quad
	\text{and}\quad
	h'q_1q_2\neq hq_1'q_2'.
\end{equation}
We have
\begin{equation} \label{eq:12.4}
	h' \xi(q_1',q_2';n) - h\xi(q_1,q_2;n)
	\equiv 
	\frac{s \overline{n}}{r}
	+ \frac{t_1 \overline{n}}{q_1}
	+ \frac{t_1' \overline{n}}{q_1'}
	+ \frac{t_2 \overline{(n+kr)}}{q_2}
	+ \frac{t_2' \overline{(n+kr)}}{q_2'}\
	 (\text{mod }1)
\end{equation}
with
\begin{align}
	s &\equiv a (h' \overline{q_1'q_2'} - h\overline{q_1q_2})\ (\text{mod }r),
	\\
	t_1 &\equiv -ah \overline{q_2 r}\ (\text{mod } q_1),
	\\
	t_1' &\equiv ah' \overline{q_2'r}\ (\text{mod } q_1'),
	\\
	t_2 &\equiv -ah \overline{q_1 r}\ (\text{mod } q_2),\\
	t_2' &\equiv ah' \overline{q_1'r}\ (\text{mod } q_2').
\end{align}
Letting $d_1=[q_1,q_1']r$ and $d_2=[q_2,q_2']$, we may rewrite \eqref{eq:12.4} as
\begin{equation}
	h' \xi(q_1',q_2';n) - h\xi(q_1,q_2;n) 
	\equiv 
	\frac{c_1 \overline{n}}{d_1}
	+ \frac{c_2 \overline{(n_kr)}}{d_2}\
	(\text{mod }1)
\end{equation}
for some $c_1$ and $c_2$ with
\begin{equation}
	(c_1,r) = (h'\overline{q_1'q_2'} - h\overline{q_1q_2},r).
\end{equation}
It follows by Lemma \ref{lemma:11} that
\begin{equation} \label{eq:12.5}
	\mathcal{M}
	\ll (d_1d_2)^{1/2+\epsilon}
	+ \frac{(c_1,d_1)(d_1,d_2)^2N}{d_1}.
\end{equation}
By \eqref{eq:7.5} and \eqref{eq:9.13} we have
\begin{equation}
	(d_1d_2)^{1/2}
	\ll (Q^4R)^{1/2}
	\ll X^{1/4+16\varpi}.
\end{equation}
On the other hand, we have $(d_1,d_2)\le (q_1q_1',q_2q_2')\ll Q^2$, since $(q_2q_2',r)=1$, and, similar to \eqref{eq:9.14},
\begin{equation}
	(c_1,d_1)
	\le (c_1,r)[q_1,q_1']
	\ll [q_1,q_1'] H_1Q^2.
\end{equation}
It follows by \eqref{eq:10.16}, \eqref{eq:9.13}, and the first inequality in \eqref{eq:9.12} that
\begin{equation}
	\frac{(c_1,d_1)(d_1,d_2)^2N}{d_1}
	\ll H_1 N Q^6 R^{-1}
	\ll X^{72\varpi}.
\end{equation}
Combining these estimates with \eqref{eq:12.5} we deduce that
\begin{equation}
	\mathcal{M}(h,h'; q_1, q_2; q_1', q_2')
	\ll X^{1/4+ 16\varpi + \epsilon}.
\end{equation}
Together with \eqref{eq:10.16}, this implies that, on the right side of \eqref{eq:12.2}, the contribution from terms with $h'q_1q_2\neq hq_1'q_2'$ is
\begin{equation}
	X^{1/4+16\varpi+\epsilon} H_1^2
	\ll X^{1/4+33\varpi},
\end{equation}
which is sharper than the right side of \eqref{eq:12.3}. Combining these estimates with \eqref{eq:12.2} we obtain
\begin{equation} \label{eq:12.6}
	\sum_{\substack{n\sim N\\ (n,r)=1}}
	|\mathcal{K}(n)|
	\ll X^{1-\varpi}.
\end{equation}
The estimate \eqref{eq:10.15} follows from \eqref{eq:12.1} and \eqref{eq:12.6} immediately. This completes the proof of Theorem \ref{thm:maintheoremtauk}.

\section{Proof of uniform power savings Theorem \ref{thm:maintheoremtaukOnLindelof}} \label{section:proofOfUniformPowerSavings}

Let $\chi$ be a primitive character (mod $d$) and $L(s,\chi)$ denote its Dirichlet $L$-function. On the Generalized Lindel\"of Hypothesis, we have, for $\sigma\ge 1/2$, 
	\begin{equation} \label{eq:LindelofBound}
	L(s,\chi)^k \ll
	(d|s|)^{\epsilon k}
	\end{equation}
for any $\epsilon>0$; see, e.g., \cite{ConreyGhosh2006}. This bound will allow us to significantly improve the estimate in \eqref{eq:dsmallpowerofx}.

This Lemma is a truncated Perron's formula.

\begin{lemma} \label{lemma:perronFormula}
	Let
	\begin{equation}
		\delta(X) = 
		\begin{cases}
		0, & \text{if } 0<X<1,\\
		1/2,& \text{if } X=1,\\
		1,& \text{if } X>1
		\end{cases}
	\end{equation}
	and
	\begin{equation}
		I(X,T)
		= \frac{1}{2\pi i}
		\int_{c- iT}^{c+ iT}
		\frac{X^s}{s} ds.
	\end{equation}
	Then
	\begin{equation}
		\delta(X) = \frac{1}{2\pi i}
		\int_{c- i\infty}^{c+ i\infty}
		\frac{X^s}{s} ds,
	\end{equation}
	and, for $X>0$, $c>0$, and $R>0$, we have
	\begin{equation}
		|I(X,T) - \delta(X)|
		< \begin{cases}
		X^c \min\{1, T^{-1}|\log X|^{-1} \},& \text{if } X\neq 1,\\
		c/T,& \text{if } X=1.
		\end{cases}
	\end{equation}
\end{lemma}
\begin{proof}
	See \cite[Theorem 4.1.4]{MurtyBook2007}.
\end{proof}

With \eqref{eq:LindelofBound} we can strengthen Lemma \ref{lemma:tinymoduli}, which was used in the proof of Theorem \ref{thm:maintheoremtauk}, to
	
\begin{prop} \label{prop:SWOnLindelof}
	Assume the Generalized Lindel\"of Hypothesis. For $\chi$ a primitive character (mod $d$) we have
	\begin{equation} \label{eq:SWOnLindelof}
	\sum_{n\le X} \tau_k(n)\chi(n) 
	\ll X^{7/8} d^{1/2}.
	\end{equation}
\end{prop}

\begin{proof}
	The proof of this proposition is in principal very similar to that of Lemma \ref{lemma:tinymoduli}. Indedd, we estimate directly the left side of \eqref{eq:SWOnLindelof} using the truncated Perron's formula, getting
	\begin{equation} \label{eq:LindelofBound3}
		\sum_{n\le X} \tau_k(n)\chi(n)
		=
		\frac{1}{2\pi i} \int_{9/8-iT}^{9/8+iT}
		L(s,\chi)^k \frac{X^s}{s} ds
		+ O\left(\frac{X^{9/8}}{T} \right).
	\end{equation}
	Since $\chi$ is nonprincipal, the function $L(s,\chi)^k$ is analytic and has no poles in $\sigma\ge 1/2$. We move the line of integration to $\sigma=1/2$ and apply Cauchy's theorem. On the generalized Lindel\"of Hypothesis, we apply \eqref{eq:LindelofBound} with $\epsilon=1/2k$ giving
	\begin{equation}
		L(s,\chi)^k \ll (d|s|)^{1/2},\quad
		\sigma \ge 1/2.
	\end{equation}
	The contribution from horizontal segments is
	\begin{equation}
		\left|\frac{1}{2\pi i} 
		\left( \int_{9/8+ iT}^{1/2+ iT} 
		+ \int_{9/8- iT}^{1/2- iT}  \right)
		L(s,\chi)^k \frac{X^s}{s} ds \right|
		\ll d^{1/2} \frac{X^{9/8}}{T^{1/2}},
	\end{equation}
	and contribution from the vertical segment $\sigma=1/2$ is
	\begin{equation}
		\left|\frac{1}{2\pi i} 
		\int_{1/2- iT}^{1/2+ iT} 
		L(s,\chi)^k \frac{X^s}{s} ds \right|
		\ll d^{1/2} X^{1/2} T^{1/2}.
	\end{equation}
	Hence, by Cauchy's theorem, \eqref{eq:LindelofBound3} becomes
	\begin{equation}
		\sum_{n\le X} \tau_k(n)\chi(n)
		\ll d^{1/2} \frac{X^{9/8}}{T^{1/2}}
		+ d^{1/2} X^{1/2} T^{1/2}
		+\frac{X^{9/8}}{T}.
	\end{equation}
	We choose $T=X^{1/2}$. Thus, the error term of the above is
	\begin{equation}
		\ll d^{1/2} X^{7/8}.
	\end{equation}
	This leads to the right side of \eqref{eq:SWOnLindelof}.
\end{proof}

\textit{Proof of Theorem \ref{thm:maintheoremtaukOnLindelof}.} By Proposition \ref{prop:2}, we have
\begin{equation} \label{eq:mediummoduli3}
	\sum_{\substack{d\in \mathcal{D}\\ D_2< d< D_3}}
	|\Delta(\tau_k;X,d,a)|
	\ll 
	X^{1 - \varpi^2}.
\end{equation}
By Proposition \ref{prop:SWOnLindelof} above together with the large sieve inequality \eqref{eq:largesieveineq}, in a way similar to the proof of Proposition \ref{prop:mediummoduli}, for $D \le D_2$, we get
\begin{equation} \label{eq:mediummoduli2}
	\sum_{d\le D}\max_{(a,d)=1}
	|\Delta(\tau_k;X,d,a)|
	\ll 
	X^{1 - \varpi^2}.
\end{equation}
This, together with \eqref{eq:mediummoduli3}, give the desired estimate
\begin{equation}\label{eq:tauk3}
	\sum_{\substack{d\in \mathcal{D}\\ d< X^{1/2+1/584}}}
	\left|
	\Delta(\tau_k;X,d,a)
	\right|
	\ll X^{1 - \varpi^2}.
\end{equation}

\section{Proofs of Theorems \ref{thm:meanSquareResult} and \ref{thm:pairOfTau_k}}

\subsection{Proof of theorem \ref{thm:meanSquareResult}}

We proceed analogously as in the proof of Lemma \ref{lemma:10} and Proposition \ref{prop:mediummoduli}. By \eqref{eq:nonprincipalsum2}, we have
\begin{equation}
\Delta({{\tau_k}; X, d,a})^2
=\frac{1}{\varphi(d)^2}
\displaystyle\sideset{}{'}\sum_{\chi_1 (\text{mod } d)}
\overline{\chi_1}(a)
\left(
\sum_{n\le X} \tau_k(n) \chi_1(n)
\right)
\displaystyle\sideset{}{'}\sum_{\chi_2 (\text{mod } d)}
\chi_2(a)
\overline{\left(
	\sum_{n\le X} \tau_k(n) \chi_2(n)
	\right)}.
\end{equation}
Summing over primitive $a (\text{mod }d)$ and changing the order of summation, we get
\begin{equation}
\displaystyle\sideset{}{^*}\sum_{a (\text{mod } d)}
\Delta({{\tau_k}; X, d,a})^2
= \frac{1}{\varphi(d)^2}
\displaystyle\sideset{}{'}\sum_{\chi_1 (\text{mod } d)}\
\displaystyle\sideset{}{'}\sum_{\chi_2 (\text{mod } d)}
{\left(
	\sum_{n\le X} \tau_k(n) \chi_1(n)
	\right)}
\overline{\left(
	\sum_{n\le X} \tau_k(n) \chi_2(n)
	\right)}
\displaystyle\sideset{}{^*}\sum_{a (\text{mod } d)}
\overline{\chi_1}(a) \chi_2(a).
\end{equation}
By the orthogonality relation \eqref{eq:928} this becomes
\begin{equation}
\displaystyle\sideset{}{^*}\sum_{a (\text{mod } d)}
\Delta({{\tau_k}; X, d,a})^2
= \frac{1}{\varphi(d)}
\left|
\displaystyle\sideset{}{'}\sum_{\chi (\text{mod } d)}
{\left(
	\sum_{n\le X} \tau_k(n) \chi(n)
	\right)}
\right|^2.
\end{equation}
We now reduce to primitive characters. By Lemma \ref{lemma:primitivesum}, we have
\begin{equation}
\sum_{d\le D}\
\displaystyle\sideset{}{^*}\sum_{a (\text{mod } d)}
\Delta({{\tau_k}; X, d,a})^2
\ll 
\log \mathcal{L}
\sum_{r\le D}
\frac{1}{r}
\left(
\sum_{1<q\le D/r}
\frac{1}{q}
\left|\
\displaystyle\sideset{}{^*}\sum_{\chi (\text{mod } q)}
\left(
\sum_{n\le X} \tau_k(n) \chi(n)
\right)
\right|^2
\right).
\end{equation}
By Lemma \ref{lemma:tinymoduli}, we get, for $1< q\le X^{1/3(k+2)}$,
\begin{equation}
\sum_{1<q\le X^{1/3(k+2)}}
\frac{1}{q}
\left|\
\displaystyle\sideset{}{^*}\sum_{\chi (\text{mod } q)}
\left(
\sum_{n\le X} \tau_k(n) \chi(n)
\right)
\right|^2
\ll \sum_{1<q\le X^{1/3(k+2)}}
\frac{1}{q}
\left(
\varphi(q) X^{2-\frac{2}{3(k+2)}}
\right)
\ll X^{2-\frac{1}{3(k+2)}}.
\end{equation}
Assume $X^{1/3(k+2)} \ll Q\ll D$. By the large sieve inequality \eqref{eq:largesieveineq} and the bound \eqref{eq:lemma81}, we have
\begin{equation}
\frac{1}{Q}
\sum_{q\sim Q}
\left| \
\displaystyle\sideset{}{^*}\sum_{\chi (\text{mod } q)}
\left(
\sum_{n\le X} \tau_k(n) \chi(n)
\right)
\right|^2
\ll \frac{1}{Q}(Q^2 + X)
\left(\sum_{n \le X} \tau_k(n) \right)^2
\ll \left(Q + \frac{X}{Q} \right) X \mathcal{L}^{k-1}.
\end{equation}
This leads to \eqref{eq:meanSquareResult}.

\subsection{Proof of theorem \ref{thm:pairOfTau_k}}
Denote
\begin{equation} \label{eq:largeSieve11}
	E(f_1,f_2; X, d, a)=
	\sum_{\substack{m,n\le X\\ m\equiv an (d)}}
	f_1(m) f_2(n)
	- \frac{1}{\varphi(d)}
	\left(\sum_{\substack{m\le X\\ (m,d)=1}} f_1(m)\right)
	\left(\sum_{\substack{n\le X\\ (n,d)=1}} f_2(n)\right).
\end{equation}
We start by writing
\begin{equation}
	\psi_1(\chi) = \sum_{m\le X} f_1(m) \chi(m)
	\quad	\text{and} \quad
	\psi_2(\chi) = \sum_{n\le X} f_2(n) \chi(n).
\end{equation}
We first prove the estimate \eqref{eq:pairOfTau_k}. Let $f_1=f_2=\tau_k$. 
By \eqref{eq:nonprincipalsum2}, we have
\begin{equation} \label{eq:largeSieve12}
	E(\tau_k,\tau_k; X, d, a)
	= \frac{1}{\varphi(d)}
	\displaystyle\sideset{}{'}\sum_{\chi (\text{mod } d)}
	\overline{\chi}(a)
	\psi_1(\chi)
	\overline{\psi_2(\chi)}.
\end{equation}
Taking the square of the modulus of both sides yields
\begin{equation}
	E(\tau_k,\tau_k; X, d, a)^2
	=\frac{1}{\varphi(d)^2}
	\displaystyle\sideset{}{'}\sum_{\chi_1 (\text{mod } d)}
	\overline{\chi_1}(a)
	\psi_1(\chi_1)
	\overline{\psi_2(\chi_1)}
	\displaystyle\sideset{}{'}\sum_{\chi_2 (\text{mod } d)}
	\chi_2(a)
	\overline{\psi_1(\chi_2)}
	\psi_2(\chi_2).
\end{equation}
Summing over primitive $a (\text{mod }d)$ and changing the order of summation, we get
\begin{equation}
	\displaystyle\sideset{}{^*}\sum_{a (\text{mod } d)}
	E(\tau_k,\tau_k; X, d, a)^2
	= \frac{1}{\varphi(d)^2}
	\displaystyle\sideset{}{'}\sum_{\chi_1 (\text{mod } d)}\
	\displaystyle\sideset{}{'}\sum_{\chi_2 (\text{mod } d)}
	\psi_1(\chi_1)
	\overline{\psi_2(\chi_1)}
	\overline{\psi_1(\chi_2)}
	\psi_2(\chi_2)
	\displaystyle\sideset{}{^*}\sum_{a (\text{mod } d)}
	\overline{\chi_1}(a) \chi_2(a).
\end{equation}
By the orthogonality relation \eqref{eq:928} we get
\begin{equation}
	\displaystyle\sideset{}{^*}\sum_{a (\text{mod } d)}
	E(\tau_k,\tau_k; X, d, a)^2
	= \frac{1}{\varphi(d)}
	\left|
	\displaystyle\sideset{}{'}\sum_{\chi (\text{mod } d)}
	\psi_1(\chi) \overline{\psi_2(\chi)} \right|^2.
\end{equation}
We now reduce to primitive characters. By Lemma \ref{lemma:primitivesum}, we have
\begin{equation}
	\sum_{d\le D}\
	\displaystyle\sideset{}{^*}\sum_{a (\text{mod } d)}
	E(\tau_k,\tau_k; X, d, a)^2
	\ll 
	\log \mathcal{L}
	\sum_{r\le D}
	\frac{1}{r}
	\left(
	\sum_{1<q\le D/r}
	\frac{1}{q}
	\left|\
	\displaystyle\sideset{}{^*}\sum_{\chi (\text{mod } q)}
	\psi_1(\chi) \overline{\psi_2(\chi)}\
	\right|^2
	\right).
\end{equation}
By Lemma \ref{lemma:tinymoduli}, we get, for $1< q\le X^{1/3(k+2)}$,
\begin{equation} \label{eq:pairOfTau_k3}
	\sum_{1<q\le X^{1/3(k+2)}}
	\frac{1}{q}
	\left|\
	\displaystyle\sideset{}{^*}\sum_{\chi (\text{mod } q)}
	\psi_1(\chi) \overline{\psi_2(\chi)}\
	\right|^2
	\ll \sum_{1<q\le X^{1/3(k+2)}}
	\frac{1}{q}
	\left(
	\varphi(q) X^{2-\frac{2}{3(k+2)}}
	\right)^2
	\ll X^{4-\frac{2}{3(k+2)}}.
\end{equation}
Assume $X^{1/3(k+2)} \ll Q\ll D$. It suffices then to show, for each fixed $r\le D$,
\begin{equation}
	\frac{1}{Q}
	\sum_{q\sim Q}
	\left| \
	\displaystyle\sideset{}{^*}\sum_{\chi (\text{mod } q)}
	\psi_1(\chi) \overline{\psi_2(\chi)}\
	\right|^2
	\ll X^{4-1/3(k+3)}.
\end{equation}
By the large sieve inequality \eqref{eq:largesieveineq} and the bound \eqref{eq:lemma81}, the left-side of the above is
\begin{equation} \label{eq:largeSieve10}
	\le \frac{1}{Q}
	\sum_{q\sim Q}\
	\displaystyle\sideset{}{^*}\sum_{\chi (\text{mod } q)}
	|\psi_1(\chi) \overline{\psi_2(\chi)}|^2
	\ll \frac{1}{Q}(Q^2 + X^2)
	\left(\sum_{n \le X} \tau_k(n) \right)^2
	\ll \left(Q + \frac{X^2}{Q} \right) X^2 \mathcal{L}^{2k-2}.
\end{equation}
For $D\le X^{2- 1/3(k+2)}$, the above is
\begin{equation}
	\ll X^{2-1/3(k+2)} + X^{2-1/3(k+2)} 
	X^{2} \mathcal{L}^{2k} 
	\ll X^{4-1/3(k+3)}.
\end{equation}
This, together with \eqref{eq:pairOfTau_k3}, lead to the right side of \eqref{eq:pairOfTau_k}. 
Note also that, for $X^{2-1/3(k+2)} < D \le X^2$, \eqref{eq:largeSieve10} becomes
\begin{equation}
	\left(D + \frac{X^2}{X^{2-1/3(k+2)}} \right)
	X^{2} \mathcal{L}^{2k-2} 
	\ll DX^{2} \mathcal{L}^{2k-2}.
\end{equation}
This gives an estimate for \eqref{eq:pairOfTau_k} in this range.

We next prove \eqref{eq:pairOfTau_k2}. Let $f_1=\tau_k$ and $f_2=\Lambda$. The proof of \eqref{eq:pairOfTau_k2} is analogous to that of \eqref{eq:pairOfTau_k}, except that in \eqref{eq:pairOfTau_k3} and \eqref{eq:largeSieve10} we estimate the sum over $\Lambda$ by
\begin{equation}
	\psi_2(\chi)
	= \sum_{n\le X} \Lambda(n) \chi(n)
	\ll X
	\quad \text{and}\quad
	\sum_{n\le X} \Lambda(n) \ll X.
\end{equation}
Thus \eqref{eq:pairOfTau_k3} becomes
\begin{equation} \label{eq:pairOfTau_k4}
	\sum_{1<q\le X^{1/3(k+2)}}
	\frac{1}{q}
	\left|\
	\displaystyle\sideset{}{^*}\sum_{\chi (\text{mod } q)}
	\psi_1(\chi) \overline{\psi_2(\chi)}\
	\right|^2
	\ll X^{4-\frac{1}{3(k+2)}},
\end{equation}
and \eqref{eq:largeSieve10} becomes
\begin{equation}
	\frac{1}{Q}
	\sum_{q\sim Q}\
	\displaystyle\sideset{}{^*}\sum_{\chi (\text{mod } q)}
	|\psi_1(\chi) \overline{\psi_2(\chi)}|^2
	\ll \frac{1}{Q}(Q^2 + X^2)
	\left(\sum_{n \le X} \tau_k(n) \right)
	\left(\sum_{n \le X} \Lambda(n) \right)
	= \left(Q + \frac{X^2}{Q} \right) X^2 \mathcal{L}^{k-1},
\end{equation}
both of which are admissible for \eqref{eq:pairOfTau_k2}. This concludes the proof of Theorem \ref{thm:pairOfTau_k}.

\bigskip
\textbf{Added in proof.} Fouvry has informed us of recent work with Radziwi\l{}\l{} \cite{FouvryRaziwill2018}, which is now to appear in ``Annales Scientifiques de l'Ecole Normale Sup\'erieure", in which the authors proved in Corollary 1.2 that for fixed integer $k\ge 1$ and $\epsilon>0$, one has
\begin{equation}
	\sum_{\substack{Q\le q\le 2Q\\ (q,a)=1}}
	\left|
	\sum_{\substack{X < n\le 2X\\ n\equiv a (\text{mod } q)}} \tau_k(n)
	- \frac{1}{\varphi(q)} \sum_{\substack{X< n\le 2X\\ (n,q)=1 }} \tau_k(n)
	\right|
	\ll \frac{X}{(\log X)^{1-\epsilon}}
\end{equation}
uniformly for $X\ge 2$, $Q\le X^{\frac{17}{33} - \epsilon}$ and $1\le |a| \le X/12$.

\end{document}